\documentclass[a4paper,oneside,english]{article}
\usepackage[T1]{fontenc}
\usepackage[latin9]{inputenc}
\usepackage{babel}
\usepackage{amsmath}
\usepackage{amsthm}
\usepackage{amstext}
\usepackage{amssymb}
\usepackage{setspace}
\usepackage{tikz}
\usepackage{graphicx} 
\usepackage{lineno}
\usepackage{float} 

\numberwithin{equation}{section}
\numberwithin{figure}{section}
\newtheorem{thm}{Theorem}[section]
\newtheorem{cor}[thm]{Corollary}
\theoremstyle{definition}

\theoremstyle{plain}

\theoremstyle{plain}
\newtheorem{prop}[thm]{Proposition}
\theoremstyle{definition}
\newtheorem{rem}[thm]{Remark}
\theoremstyle{plain}

\theoremstyle{plain}
\newtheorem{assumptionA}{Assumption}
\theoremstyle{plain}

\newtheorem{lemma}[thm]{Lemma}
\theoremstyle{plain}

\newcommand{\E}{\mathbb{E}}
\usepackage{hyperref}
\usepackage{url}

\makeatother
\begin{document}

\title{Uniform sampling in a structured branching population}

\author{Aline Marguet\footnote{Univ. Grenoble Alpes, Inria, 38000 Grenoble, France, aline.marguet@inria.fr}}
\maketitle
\date{}

\begin{abstract}
We are interested in the dynamic of a structured branching population where the trait of each individual moves according to a Markov process. The rate of division of each individual is a function of its trait and when a branching event occurs, the trait of the descendants at birth depends on the trait of the mother and on the number of descendants. In this article, we explicitly describe the penalized Markov process, named auxiliary process, corresponding to the dynamic of the trait of a "typical" individual by giving its associated infinitesimal generator. We prove a Many-to-One formula and a Many-to-One formula for forks.
 Furthermore, we prove that this auxiliary process characterizes exactly the process of the trait of a uniformly sampled individual in a large population approximation. We detail three examples of growth-fragmentation models: the linear growth model, the exponential growth model and the parasite infection model.
\end{abstract}

\paragraph*{A.M.S classification:}
60J80, 60J85, 60J75, 92D25.

\section{Introduction}
The characterization of the sampling of individuals in a population is a key issue for branching processes with several motivations in statistics and biology. We refer to the work of Durrett \cite{durrett1978genealogy} and references therein for the study of the genealogy of a branching Markov process and the study of the degree of relationship between $k$ individuals chosen randomly at time $t$ in the population. In particular, he analyzed the asymptotics of the so-called reduced branching process $N_t(s)$ defined as the number of individuals alive at time $s$ which have offspring alive at time $t$. An approximation of this process by a pure birth process is given in \cite{o1995genealogy}. The question of finding the coalescing time of individuals in a Galton-Watson tree is addressed in \cite{zubkov1976limiting} and the coalescent structure of continuous-time Galton-Watson trees is studied in \cite{harris2017coalescent}. We refer to \cite{athreya2012coalescence} and \cite{lambert2013coalescent} for more results on this question and to \cite{hong2011coalescence} for results concerning the Bellman-Harris branching process. The pedigree of a typical individual in a supercritical branching process has also been investigated asymptotically for multi-type branching processes with a finite number of types in \cite{georgii2003supercritical}, with i.i.d lifetimes in \cite{athreya2011} and with an age-structure in \cite{nerman1984stable}. The characterization of the sampling is the key to obtain asymptotic results on the branching process (\cite{kurtz1997conceptual}, \cite{bansaye2011limit}, \cite{cloez}) and to infer the parameters of the model (\cite{guyon2007limit}, \cite{doumic2015statistical}, \cite{hoffmann2015nonparametric}). 

In this article, we consider a continuous-time structured branching Markov process where the trait of each individual moves according to a Markov process and influences the branching events.
The purpose of this article is to characterize the trait of a typical individual uniformly sampled from the population at time $t$ and its associated ancestral lineage. In particular, we exhibit the bias due to the structure of the population and to the sampling. We also describe the traits of a uniformly sampled couple in the current population. Therefore, we provide new applications in a non-neutral framework for cell division (Section \ref{sec:exdebut}), even for models in a varying environment.

We now describe informally the process, while its rigorous construction and characterization as a c\`adl\`ag measure-valued process under Assumptions \ref{assu:debut1} and \ref{assu:debut2} are detailed in Section \ref{sec:exis}. We assume that individuals behave independently and that for each individual $u$:
\begin{itemize}
\item its trait $(X_t^u)_{t\geq 0}$ evolves as an $\mathcal{X}$-valued Markov process with infinitesimal generator $\left(\mathcal{G},\mathcal{D}(\mathcal{G})\right)$, where $\mathcal{X}\subset\mathbb{R}^d$ is a measurable space for some $d\geq 1$,
\item it dies at time $t$ at rate $B(X_t^u)$,
\item at death, an individual with trait $x$ is replaced by $k\in\mathbb{N}$ individuals with probability $p_k(x)$ and $m(x)=\sum_{k\geq 1}kp_k(x)$,
\item the trait of the $j$th child among $k$ is distributed as $P_j^{(k)}(x,\cdot)$ for all $1\leq j\leq k$. 
\end{itemize}

We use the notion of spine, which is a distinguished line of descent in the branching process, and Many-to-One formulas, which have been developed from the notion of size-biased tree, considered by Kallenberg \cite{kallenberg1977stability}, Chauvin and Rouault \cite{chauvin1988kpp}, Chauvin, Rouault and Wakolbinger \cite{chauvin1991growing} with a Palm measure approach and Lyons, Peres and Pemantle \cite{lyons1995conceptual}. For general results on branching processes using these techniques, including the spinal decomposition, we refer to \cite{kurtz1997conceptual} and \cite{athreya2000change} for discrete-time models and to \cite{georgii2003supercritical}, \cite{hardy2009spine} and \cite{cloez} for continuous-time branching processes. These previous works ensure in particular that if we denote by $V_t$ the set of individuals alive at time $t$ and by $N_t$ its cardinal, we have the well-known Many-to-One formula:
\begin{align}\label{eq:mto_poid}
\mathbb{E}\left[\sum_{u\in V_t} f\left(X_t^u\right)\right]=\mathbb{E}\left[f\left(Y_t\right)e^{\int_0^t B(Y_s)(m(Y_s)-1)ds}\right],
\end{align}
where $f$ is a non-negative measurable function and $(Y_t)_{t\geq 0}$ follows the dynamic of a tagged-particle i.e. the same dynamic of all the particles between jumps and at a jump, the unique daughter particle is chosen uniformly at random among all the daughter particles. This formula can be seen as a Feynman-Kac formula  \cite[Section 1.3]{del2004feynman} with a weight on the right-hand side relying on the whole ancestral lineage of current individuals which corresponds to the growth of the population. In this case, under spectral assumptions, the asymptotic behavior of the number of individuals has been well studied in \cite{lyons1995conceptual}, \cite{kurtz1997conceptual}, \cite{athreya2000change}, \cite{georgii2003supercritical} and \cite{biggins2004measure}. We also refer to the work of Bansaye and al. \cite{bansaye2011limit} for law of large numbers theorems using Many-to-One formulas.

On the right-hand side of \eqref{eq:mto_poid} appears a Markov process with penalized (or rewarded) trajectories which describes the dynamic of the trait of a typical individual. This corresponds to a time-inhomogeneous Markov process $Y^{(t)}$, indexed by $t\geq 0 $, for which we provide the following formula for any non-negative measurable function $F$ on the space of c\`adl\`ag processes:
\begin{align}\label{eq:mto_intro}
\mathbb{E}\left[\sum_{u\in V_t} F\left(X_s^u,s\leq t\right)\right]=m(x,0,t)\mathbb{E}\left[F\left(Y_s^{(t)},s\leq t\right)\right],
\end{align} 
where for $x\in\mathcal{X}$ and $0\leq s\leq t$,
\begin{align}\label{def:mt}
m(x,s,t):=\mathbb{E}\left[N_t\big|Z_s=\delta_x\right],
\end{align}
and 
\begin{align*}
Z_t=\sum_{u\in V_t} \delta_{X_t^u},
\end{align*}
is the empirical measure of the process. We explicit the generator $\left(\mathcal{A}_s^{(t)}\right)_{s\leq t}$ of this auxiliary process: for all well-chosen functions $f$, $x\in\mathcal{X}$ and $s<t$, we have
\begin{align*}
\mathcal{A}_s^{(t)}f(x)=\widehat{\mathcal{G}}_{s}^{(t)}f(x)+\widehat{B}_{s}^{(t)}(x)\int_{\mathcal{X}}(f(y)-f(x))\widehat{P}_{s}^{(t)}(x,dy),
\end{align*}
where
\begin{align*}
\widehat{\mathcal{G}}_{s}^{(t)}f(x)=\frac{\mathcal{G}\left(m(\cdot,s,t)f\right)(x)-f\left(x\right)\mathcal{G}\left(m(\cdot,s,t)\right)(x)}{m(x,s,t)},
\end{align*}
\begin{align*}
\widehat{B}_{s}^{(t)}(x)=B(x)\int_{\mathcal{X}}\frac{m(y,s,t)}{m(x,s,t)}m(x,dy),
\end{align*}
\begin{align*}
\widehat{P}_{s}^{(t)}(x,dy)=m(y,s,t)m(x,dy)\left(\int_{\mathcal{X}}m(y,s,t)m(x,dy)\right)^{-1},
\end{align*}
and
\[
m(x,A):=\sum_{k\geq 0}p_{k}(x)\sum_{j=1}^{k}P_{j}^{(k)}(x,A),
\]
denotes the expected number of children with trait in the Borel set $A$ of an individual with trait $x$.

Moreover, we give some very simple and interesting examples where we can find the expression of the generator of the auxiliary process: we detail three models for the dynamic of a cell population (see Section \ref{sec:exdebut}).

 The Many-to-One formula \eqref{eq:mto_intro} splits the behavior of the entire population into a term characterizing the growth of the population and a term characterizing the dynamic of the trait. This separation in two terms is the key to the study of the ancestral trait of a uniformly sampled individual. Indeed, we prove in Theorem \ref{th:grdpop}, that the auxiliary process describes the ancestral lineage of a sampled individual in a branching population at a fixed time when the initial population is large. More precisely, if we denote by $X^{U(t),\nu}$ the trait of a uniformly sampled individual from a population at time $t$ with initial distribution $\nu$ and if $\nu_n=\sum_{i=1}^n \delta_{X_i}$ where $X_i$ are i.i.d. random variables with law $\nu$, under some assumptions, we prove the following convergence in law:
\begin{align}\label{eq:conv_intro}
 X_{[0,t]}^{U(t),\nu_n}\underset{n\rightarrow +\infty}{\longrightarrow} Y_{[0,t]}^{(t),\pi_t},\text{ where }
\pi_t(dx)=\frac{\E(N_t\big| Z_0 = \delta_x)\nu(dx)}{\int \E(N_t\big| Z_0 = \delta_x)\nu(dx)},
\end{align}
and $Y^{(t),\pi_t}$ denotes the auxiliary process with initial condition distributed as $\pi_t$.
This result shows that the auxiliary process is the appropriate tool for the study of the trait along the ancestral lineage of a sampling. We notice in particular that the dependence of the average number of individuals in the population on the trait plays a crucial part in the creation of a bias.

Finally, we refer the reader to \cite{marguet2017law} for results on the asymptotic behavior of the process of a sampling. In particular, under some assumptions ensuring the ergodicity of the auxiliary process, a law of large numbers for the empirical distribution of ancestral trajectories is proven. The asymptotic behavior of the process of a sampling has already been studied in \cite{bansaye2011limit} in the case of a constant division rate and in \cite{cloez} in a spectral framework.

\paragraph*{Outline.}
Section \ref{existence} is devoted to the rigorous construction of our process. In Section \ref{sub:existence}, we first describe in detail the model and in Theorem \ref{th:existence}, we prove the existence and uniqueness of the branching process. Then, in Section \ref{sec:exdebut}, we introduce our three examples of cell division models: the size-structured model with linear or exponential growth and the parasite infection model. In Section \ref{sampling}, we detail the properties of the Markov process along the spine. In particular, in Theorem \ref{th:mto}, we prove the Many-to-One formula which describes the dynamic of a typical individual in the population. Finally, we give two other Many-to-One formulas, one for the dynamic of the whole tree in Proposition \ref{prop:mtotree} and an other one for the dynamic of a couple of traits in Proposition \ref{prop:mtoforksgene}. Section \ref{infiniteparticle} concerns the ancestral lineage of a uniform sampling at a fixed time in a large population. More precisely, in Theorem \ref{th:grdpop}, we prove the convergence \eqref{eq:conv_intro}. In Section \ref{sec:exfin}, we give explicitly the dynamic of the auxiliary process for our three examples of cell population models. Finally, in Section \ref{sec:other}, we give some useful comments on the model and some additional examples. 

\paragraph*{Notation.}We use the classical Ulam-Harris-Neveu notation to identify each individual. Let
\[\mathcal{U}=\bigcup_{n\in\mathbb{N}}\left(\mathbb{N}^{*}\right)^{n}.
\]
The first individual is labeled by $\emptyset $. When an individual $u\in\mathcal{U}$ dies, its $K$ descendants are labeled $u1,\ldots  ,uK$. If $u$ is an ancestor of $v$, we write $u\leq v$. 

We will denote by $\mathcal{C}^1(\mathcal{X})$ and $\mathcal{C}^2(\mathcal{X})$, the set of continuously differentiable and twice continuously differentiable functions on $\mathcal{X}$, respectively. Finally, for any stochastic process $X$ on $\mathcal{X}$ or $Z$ on the set of point measures on $\mathcal{X}$, we will denote by $\mathbb{E}_x\left[f(X_t)\right]=\E\left[f(X_t)\big|X_0 = x\right]$ and $\mathbb{E}_{\delta_x}\left[f(Z_t)\right]=\E\left[f(Z_t)\big|Z_0 = \delta_x\right]$.

\section{Definition and existence of the structured branching process\label{existence}}\label{sec:exis}
First, we introduce some useful notations and objects to characterize the branching process. Then, we prove the existence and uniqueness of the measure-valued branching process from scratch in Section \ref{sub:existence}. Henceforth, we work on a probability space denoted by $\left(\Omega,\mathcal{F},\mathbb{P}\right)$. 

\paragraph*{Dynamic of the trait.} Let $\mathcal{X}=\mathcal{Y}\times \mathbb{R}_+$ where $\mathcal{Y}\subset\mathbb{R}^d$ is a measurable space for some $d\geq 1$. It is the state space of the Markov process describing the trait of the individuals. The second component, with values in $\mathbb{R}_+$, is a time component. We assume that $\left(A_t,t\geq 0\right)$ is a strongly continuous contraction semi-group with associated infinitesimal generator $\mathcal{G}:\mathcal{D}(\mathcal{G})\subset \mathcal{C}_b(\mathcal{X})\rightarrow \mathcal{C}_b(\mathcal{X})$, where $\mathcal{C}_b(\mathcal{X})$ denotes the space of continuous bounded function from $\mathcal{X}$ to $\mathbb{R}$.

Then, according to Theorem 4.4.1 in \cite{ethier2009markov}, there is a unique solution to the martingale problem associated with $(\mathcal{G},\mathcal{D}(\mathcal{G}))$, denoted by $(X_t,t\geq 0)$. It is an $\mathcal{X}$-valued c\`adl\`ag strong Markov process. For all $0\leq s\leq t$, $x\in\mathcal{X}$,  we denote by $\Phi(x,s,t)$ the corresponding stochastic flow i.e. $(\Phi(x,s,t),t\geq s)$ is the unique solution of the martingale problem associated with $\left(\mathcal{G},\mathcal{D}(\mathcal{G})\right)$ satisfying $\Phi(x,s,s)=x$. We have the following properties:
 \begin{itemize}
 \item for all $f\in\mathcal{D}(\mathcal{G})$, $0\leq s\leq t$ and $x\in\mathcal{X}$,
 \begin{align}\label{eq:mart_prob}
 f\left(\Phi(x,s,t)\right)-f(x)-\int_{s}^t \mathcal{G}f\left(\Phi(x,s,r)\right)dr,
 \end{align}
 is a $\sigma(X_t,t\geq 0)$-martingale where $\sigma(X_t,t\geq 0)$ is the natural filtration associated with $X$.
 \item for each $0\leq s\leq t$, $\Phi(\cdot,s,t)$ is a measurable map from $\mathcal{X}$ to $\mathcal{X}$,
 \item for each $0\leq r\leq s\leq t$ and all $x\in\mathcal{X}$, $\Phi\left(\Phi(x,r,s),s,t\right)=\Phi(x,r,t),$ almost surely.
 \end{itemize}
We refer the reader to \cite{kunita1997stochastic} for more properties on stochastic flows.
\begin{rem}\label{rem:densite} According to the Hille-Yoshida theorem (see \cite[Theorem 1.2.6]{ethier2009markov}), $\mathcal{D}(\mathcal{G})$ is dense in $\mathcal{C}_b(\mathcal{X})$ for the topology of uniform convergence.
\end{rem}
\paragraph*{Division events.} An individual with trait $x$ dies at an instantaneous rate $B(x)$, where $B$ is a continuous function from $\mathcal{X}$ to $\mathbb{R}_+$. It is replaced by $A_{u}(x)$ children, where $A_{u}(x)$ is a $\mathbb{N}$-valued random variable  with distribution $\left(p_{k}\left(x\right),k\geq 0\right)$. For convenience, we assume that $p_1(x)\equiv 0$ for all $x\in\mathcal{X}$. The trait at birth of the $j$th descendant among $k$ is given by the random variable $F_j^{(k)}(x,\theta)$, where $\left(F_j^{(k)}(\cdot,\cdot),j\leq k,k\in\mathbb{N}\right)$ is a family of measurable functions from $\mathcal{X}\times [0,1]$ to $\mathcal{X}$ and $\theta$ is a uniform random variable on $[0,1]$. This formalism will prove useful for the use of Poisson point measures. For all $k\in\mathbb{N}$, let $P^{(k)}(x,\cdot)$  be the probability measure on $\mathcal{X}^k$ corresponding to the trait distribution at birth of the $k$ descendants of an individual with trait $x$. We denote by $P_{j}^{(k)}\left(x,\cdot\right)$ the $j$th marginal distribution of $P^{(k)}$ for all $k\in\mathbb{N}$ and $j\leq k$ i.e. for all Borel sets $A\subset\mathcal{X}$, we have $P_{j}^{(k)}\left(x,A\right)=P^{(k)}\left(x,\mathcal{X}^{j-1}\times A\times\mathcal{X}^{k-j}\right)$.

We denote by $\mathcal{M}_{P}(\mathcal{X})$ the set of point measures on $\mathcal{X}$. Following Fournier and M\'el\'eard \cite{fournier2004microscopic}, we work in $\mathbb{D}\left(\mathbb{R}_{+},\mathcal{M}_{P}\left(\mathcal{X}\right)\right)$, the state of c\`adl\`ag measure-valued processes. For any $\bar{Z}\in \mathbb{D}\left(\mathbb{R}_{+},\mathcal{M}_{P}\left(\mathcal{U}\times\mathcal{X}\right)\right)$, we write $\bar{Z}_t=\sum_{u\in V_t}\delta_{\left(u,X_t^u\right)}(du,dx)$ and
\[
Z_t=\sum_{u\in V_t} \delta_{X_t^u},\ t\geq 0,
\]
the marginal measure of $\bar{Z}_t(du,dx)$ on $\mathcal{X},$ where $V_t$ represents the set of individuals alive at time $t$. We set $N_t=\# V_t$. Moreover, for any process $\bar{Z}\in\mathbb{D}\left(\mathbb{R}_{+},\mathcal{M}_{P}\left(\mathcal{U}\times \mathcal{X}\right)\right)$, we define recursively the associated sequence of jump times by \[T_0(\bar{Z})=0\text{ and }T_{k+1}(\bar{Z})=\inf\left\lbrace t>T_k(\bar{Z}),\ N_t\neq N_{T_k(\bar{Z})}\right\rbrace,\]
with the standard convention that $\inf\left\lbrace\emptyset\right\rbrace=+\infty$. 
%Note that for all $k\geq 0$, $T_k(\bar{Z})=T_k(\bar{Z})$.

In order to ensure the non-explosion in finite time of such a process, we need to consider two sets of hypotheses. The first one controls what happens regarding divisions (in term of rate of division and	of mass creation).
\begin{assumptionA}\label{assu:debut1}
We consider the following assumptions:
\begin{enumerate}
\item There exist $b_1,b_2\geq 0$ and $\gamma\geq 1$ such that for all $x\in\mathcal{X}$,
\[B(x)\leq b_1\left| x\right|^{\gamma}+b_2.\]
\item  For all $t\geq 0$, there exists $\ell(t)\in\mathbb{R}_+$, increasing in $t$, such that for all $x=(y,t)\in\mathcal{X}$, $k\in\mathbb{N}$ and $\theta\in[0,1]$,
\[
\sum_{i=1}^{k} F_i^{(k)}(x,\theta)\leq x\vee \ell(t),\text{ componentwise}.
\]
\item
There exists $\overline{m}\geq 0$ such that for all $x\in \mathcal{X}$, 
\[ m(x)=\sum_k kp_k(x)\leq \overline{m}.\]

\item For all $x\in\mathcal{X}$ and $s\geq 0$, we have
\begin{align*}
\lim_{t\rightarrow+\infty}\int_{s}^{t}B\left(\Phi\left(x,s,r\right)\right)dr=+\infty,\ \text{almost surely}.
\end{align*}
\end{enumerate}
\end{assumptionA}
The first point controls the lifetimes of individuals via the division rate. In particular, if $\gamma=0$, $B$ is bounded and the non-explosion in finite time of the number of individuals in the previously defined process is obvious. In more general framework, we have to consider the other points of Assumption \ref{assu:debut1} in order to prove the non-explosion in finite time. The second point of Assumption \ref{assu:debut1} means that we consider a fragmentation process with a possibility of mass creation at division when the mass is small enough. In particular, clones are allowed in the case of bounded traits and bounded number of descendants and any finite type branching structured process can be considered. The dependence in $t$ of the threshold $\ell$ allows us to consider models in a varying environment. The last point of Assumption \ref{assu:debut1} ensures that each individual divides after a certain time.

We make a second assumption to control the behavior of traits between divisions.
\begin{assumptionA}\label{assu:debut2}
There exists a sequence of functions $(h_{n,\gamma})_{n\in\mathbb{N}}$ such that for all $n\in\mathbb{N},\ h_{n,\gamma}\in \mathcal{D}(\mathcal{G})$ and $\lim_{n\rightarrow+\infty} h_{n,\gamma}(x)=|x|^{\gamma}$ for all $x\in\mathcal{X}$ and there exist $c_1,c_2\geq 0$ such that, for all $x\in\mathcal{X}$,
\begin{align*}
\lim_{n\rightarrow+\infty}\mathcal{G}h_{n,\gamma}(x)\leq c_1|x|^{\gamma}+c_2,
\end{align*}
where $\gamma$ is defined in the first item and for $x\in\mathcal{X}, |x|^{\gamma}=\left(\sum_{i=1}^{d+1} |x_i|\right)^{\gamma}$.
\end{assumptionA}

Assumptions \ref{assu:debut1}(1) and \ref{assu:debut2} are linked via the parameter $\gamma$ which controls the balance between the growth of the population and the dynamic of the trait. The sequence of functions $(h_{n,\gamma}, n\in\mathbb{N})$ allows us to consider dynamics for the trait for which the domain of the generator does not contain the function $x\mapsto |x|^{\gamma}$.  

\subsection{Existence and uniqueness of the structured branching process\label{sub:existence}}
We now prove the strong existence and uniqueness of the structured branching process. Let $E=\mathcal{U}\times\mathbb{R}_{+}\times\left[0,1\right]\times\left[0,1\right]$ and $M\left(ds,du,dz,dl,d\theta\right)$ be a Poisson point measure on $\mathbb{R}_{+}\times E$ with intensity $ds\otimes n(du)\otimes dz\otimes dl\otimes d\theta$, where $n(du)$ denotes the counting measure on $\mathcal{U}$. Let $\left(\Phi^u\right)_{u\in\mathcal{U}}$ be a family of independent stochastic flows satisfying \eqref{eq:mart_prob} describing the individual-based dynamics. We assume that $M$ and $\left(\Phi^u\right)_{u\in\mathcal{U}}$ are independent. We denote by $\mathcal{F}_t$ the filtration generated by the Poisson point measure $M$ and the family of stochastic flows $(\Phi^u(x,s, t), u\in\mathcal{U},x\in\mathcal{X},s\leq t)$ up to time $t$. 

For all $x\in\mathcal{X}$, there exists a function $G(x,\cdot):[0,1]\rightarrow \mathbb{N}$ such that
\[G(x,l)\overset{d}{=}\left(p_k(x),\ k\in\mathbb{N}\right),\]
where $l$ is a uniform random variable on $[0,1]$. This formalism will prove useful in the use of Poisson point measure to describe the jumps in the measure-valued branching process. For convenience, for all $x\in\mathcal{X}$ and $\theta,l$ uniform random variables on $[0,1]$, we write
\begin{align*}\ F_i(x,l,\theta)=F_i^{(G(x,l))}(x,\theta).
\end{align*}
%We denote by $(\mathcal{G}_s,s\geq 0)$ the linear operator acting on a subspace of $\mathcal{C}_b(\mathcal{X})$ and we say that $f\in\mathcal{D}(\mathcal{G}_s)$ if $$\mathcal{G}_sf(x):=\lim_{t\downarrow s}\frac{\mathbb{E}\left[f(X_t)|X_s=x\right]-f(x)}{t-s}$$ exists for all $x\in\mathcal{X}$. Then, we extend the definition of $\mathcal{D}(\mathcal{G})$ to a subspace of $\mathcal{C}_b(\mathbb{R}_+\times\mathcal{X})$ as follows: a function $f\in\mathcal{C}_b(\mathbb{R}_+F\times\mathcal{X})$ is said to belong to $\mathcal{D}(\mathcal{G})$ if for every $s\geq 0$, $x\mapsto f(s,x)$ is in $\mathcal{D}(\mathcal{G}_s)$ and if $x\mapsto \mathcal{G}_sf(s,\cdot)(x)$ is in $\mathcal{C}_b(\mathcal{X})$ (for details about space-time processes, see \cite{van2011markov}, Chapter 4).
We also define $$\bar{\mathcal{D}}(\mathcal{G}):=\left\lbrace f:\mathcal{U}\times \mathbb{R}_+\times\mathcal{X}\rightarrow\mathbb{R} \text{ such that }f(u,s,\cdot)\in\mathcal{D}(\mathcal{G}),\  \forall u\in\mathcal{U},\ s\geq 0\right\rbrace.$$
For all $0\leq s \leq t$, $f\in\bar{\mathcal{D}}(\mathcal{G})$, $x\in\mathcal{X}$ and $u\in\mathcal{U}$, we consider the $\mathcal{F}_t$-martingale $\left(M_{s,t}^{f,u}(x),t\geq s\right)$ defined by
\begin{align}\nonumber
M_{s,t}^{f,u}(x):= & f(u,t,\Phi^u(x,s,t))-f(u,s,x)\\\label{eq:probmart}
&-\int_{s}^t\left(\mathcal{G}f(u,r,\Phi^u(x,s,r))+\partial_rf(u,r,\Phi^u(x,s,r))\right)dr.
\end{align}

\begin{thm}\label{th:existence}
Under Assumptions \ref{assu:debut1}(1-3) and \ref{assu:debut2}, there exists a strongly unique $\mathcal{F}_t$-adapted c\`adl\`ag process $(\overline{Z}_s, s\geq 0)$ taking values in $\mathcal{M}_P(\mathcal{U}\times\mathcal{X})$ such that, for all $f\in \bar{\mathcal{D}}(\mathcal{G})$ and $t\geq 0$,
\begin{multline}
\langle \bar{Z}_{t},f\rangle = f\left(\emptyset,0,x_{0}\right)+\int_{0}^{t}\int_{\mathcal{U}\times\mathcal{X}}\left(\mathcal{G}f(u,s,x)+\partial_s f\left(u,s,x\right)\right)\bar{Z}_{s}\left(du,dx\right)ds+M_{0,t}^f(x)\\
 +\int_{0}^{t}\int_{E}\mathbf{1}_{\left\{u\in V_{s^{-}},\ z\leq B\left(X_{s^{-}}^u\right)\right\}}
 \left(\sum_{i=1}^{G\left(X_s^u,l\right)}f\left(u,s,F_{i}\left(X_s^u,l,\theta\right)\right)-f\left(u,s,X_{s^{-}}^{u}\right)\right)\\
 \times M\left(ds,du,dz,dl,d\theta\right),\label{eq:evol}
\end{multline}
where, for all $s\geq 0$ and $t\geq s$,
\begin{align*}
M_{s,t}^f(x)=\sum_{k\geq 1}\mathbf{1}_{s\leq T_{k-1}(\bar{Z})<t}\sum_{u\in V_{T_{k-1}(\bar{Z})}} M_{T_{k-1}(\bar{Z}),T_k(\bar{Z})\wedge t}^{f,u}\left(X^u_{T_{k-1}(\bar{Z})}\right),
\end{align*}
is a $\mathcal{F}_t$-martingale.
\end{thm}
The existence and uniqueness of such measure-valued process has first been studied by Fournier and M\'el\'eard \cite{fournier2004microscopic}. We also refer to \cite{tran,bansaye2015stochastic} for different extensions and to \cite{BT} for the case of branching processes. Here, we obtain the non-explosion of the branching process in finite time under quite general assumptions (no bounded branching rate, random number of offspring, random transmission of the trait). 

The proof of this theorem is split into four lemmas. First, in Lemma \ref{lemma:existence}, we prove the existence of a $\mathcal{F}_t$-adapted c\`adl\`ag measure-valued process $\bar{Z}$ solution of \eqref{eq:evol} for all $t\in [0,T_k(\bar{Z}))$ and all $k\in\mathbb{N}$. Then, in Lemma \ref{lemma:martingale}, we prove that $\left(M_{s,t}^f,t\geq 0\right)$ is a $\mathcal{F}_t$-martingale. Next, in Lemma \ref{lemma:uniqueness}, we prove the uniqueness of the increasing sequence $(T_k(\bar{Z}),\ k\geq 0)$ corresponding to the jump times of a solution $\bar{Z}$ to \eqref{eq:evol} and the uniqueness of a $\mathcal{F}_t$-adapted c\`adl\`ag solution to \eqref{eq:evol} for $t\in[0,T_k(\bar{Z}))$ for all $k\in\mathbb{N}$. Finally, in Lemma \ref{lemma:nonexpl}, we prove that the sequence of jump times tends to infinity resulting in the existence and uniqueness of the process on $\mathbb{R}_+$.

\begin{lemma}\label{lemma:existence}
There exists a $\mathcal{F}_t$-adapted c\`adl\`ag measure-valued process $(\bar{Z}_t,\ t\geq 0)\in\mathcal{M}_P(\mathcal{U}\times \mathcal{X})$ which is solution of \eqref{eq:evol} for all $f\in\bar{\mathcal{D}}(\mathcal{G})$ and for all $t\in[0,T_k(\bar{Z}))$, $k\in\mathbb{N}$.
\end{lemma}

\begin{proof} See Section \ref{app:existence} in the appendix.
\end{proof}
The existence of such processes has already been studied in \cite{BT} in the case of a trait following a Feller diffusion. From Lemma \ref{lemma:existence}, we deduce the existence of a c\`adl\`ag measure-valued process $Z\in\mathcal{M}_P(\mathcal{X})$ solution of \eqref{eq:evol} which is given by the projection of the solution $\bar{Z}\in\mathcal{M}_P(\mathcal{U}\times \mathcal{X})$ on the second coordinate.

\begin{lemma}\label{lemma:martingale}
Let $\bar{Z}\in\mathcal{M}_P(\mathcal{U}\times\mathcal{X})$ be a solution of \eqref{eq:evol} whose construction is given in the previous lemma. Let $k\in\mathbb{N}$. For all $0\leq s\leq t\leq T_k(\bar{Z})$, $x\in\mathcal{X}$ and $f\in \bar{\mathcal{D}}(\mathcal{G})$,
\[M_{s,t}^f(x)=\sum_{k\geq 1}\mathbf{1}_{\left\{s\leq T_{k-1}(\bar{Z})<t\right\}}\sum_{u\in V_{T_{k-1}}} M_{T_{k-1}(\bar{Z}),T_k(\bar{Z})\wedge t}^{f,u}\left(X^u_{T_{k-1}(\bar{Z})}\right),\]
is an $\mathcal{F}_t$-martingale.
\end{lemma}
\begin{proof}
Let $k\in\mathbb{N}$ and $0\leq s\leq t\leq T_k(\bar{Z})$. Let $f\in \bar{\mathcal{D}}(\mathcal{G})$ and $x\in\mathcal{X}$. Then, for all $s\leq r\leq t$, we have
\begin{align*}
&\mathbb{E}\Big[M_{s,t}^f(x)\big|\mathcal{F}_r\Big]-M_{s,r}^f(x)\\
=&\mathbb{E}\Big[\sum_{k\geq 1}\mathbf{1}_{\left\{r\leq T_{k-1}(\bar{Z})<t\right\}}\sum_{u\in V_{T_{k-1}(\bar{Z})}} M_{T_{k-1}(\bar{Z}),T_k(\bar{Z})\wedge t}^{f,u}\left(X^u_{T_{k-1}(\bar{Z})}\right)\Big|\mathcal{F}_r\Big]\\
=&\mathbb{E}\Big[\sum_{k\geq 1}\mathbf{1}_{\left\{r\leq T_{k-1}(\bar{Z})<t\right\}}\sum_{u\in V_{T_{k-1}(\bar{Z})}} \mathbb{E}\left[M_{T_{k-1}(\bar{Z}),T_k(\bar{Z})\wedge t}^{f,u}\left(X^u_{T_{k-1}(\bar{Z})}\right)\Big|\mathcal{F}_{T_{k-1}(\bar{Z})}\right]\Big|\mathcal{F}_r\Big]=0,
\end{align*}
because $\left(M_{s,t}^{f,u}(x),t\geq s\right)$ is a $\mathcal{F}_t$-martingale.
\end{proof}
Next, we prove the uniqueness of the sequence of jump times $(T_k(\bar{Z}),\ k\geq 0)$ associated with a solution $\bar{Z}\in\mathcal{M}_P(\mathcal{U}\times\mathcal{X})$ to \eqref{eq:evol} and the uniqueness of the solution on $[0,T_k(\bar{Z}))$, for all $k\in\mathbb{N}$. We refer to \cite{tran} for similar results.

\begin{lemma}\label{lemma:uniqueness}
The increasing sequence $(T_k(\bar{Z}),\ k\geq 0)$ corresponding to the jump times of a solution $\bar{Z}$ to \eqref{eq:evol} is strongly unique. Moreover, the strong uniqueness of a $\mathcal{F}_t$-adapted c\`adl\`ag measure-valued solution to \eqref{eq:evol} holds, for $t\in[0,T_k(\bar{Z}))$ and for all $k\in\mathbb{N}$.
\end{lemma}

\begin{proof}
See Section \ref{app:uniqueness} in the appendix.
\end{proof}
 
\begin{lemma}\label{lemma:nonexpl}
Under Assumptions \ref{assu:debut1}(1-3) and \ref{assu:debut2}, the strongly unique sequence of jump times of a solution $\bar{Z}$ to \eqref{eq:evol} tends to infinity as $k$ tends to infinity, almost surely.
\end{lemma}
\begin{proof}
Let $T>0$. To shorten notation, we write $T_k$ instead of $T_k(\bar{Z})$. We prove that almost surely there is no accumulation of jumps on $[0,T]$ of the solution of \eqref{eq:evol} previously constructed on $[0,T_k[$, for all $k\in\mathbb{N}$. Let $k\in\mathbb{N}$ and $\left(\bar{Z}_t,t\leq T_k\right)$ be the solution of \eqref{eq:evol} up to the $k$th division time. Using equation \eqref{eq:evol} applied to the constant function equal to $1$, for all $t\leq T_k\wedge T$, we have
\begin{align}\nonumber
&\mathbb{E}_{\delta_x}\left(N_t\right)=1+\int_0^t \mathbb{E}_{\delta_x}\left(\sum_{u\in V_s} B\left(X_s^u\right)\left(m(X_s^u)-1\right)\right)ds\\\label{eq:nbrmoyen}
&\leq  1+\overline{m}b_1\int_0^t \mathbb{E}_{\delta_x}\left(\sum_{u\in V_s} \left| X_s^u\right|^{\gamma}\right)ds+\overline{m}b_2\int_0^t \mathbb{E}_{\delta_x}\left(N_s\right)ds,
\end{align}
where the inequality comes from Assumption \ref{assu:debut1}(1) and \ref{assu:debut1}(3).
Next, using \eqref{eq:evol} again, we have
\begin{align*}
&\mathbb{E}_{\delta_x}\Bigg[\sum_{u\in V_t}h_{n,{\gamma}}(X_t^u)\Bigg]=h_{n,{\gamma}}(x) +\int_0^t\mathbb{E}_{\delta_x}\left[\sum_{u\in V_s} \mathcal{G}h_{n,{\gamma}}(X_s^u)\right]ds\\
&+\int_0^t\int_{[0,1]}\mathbb{E}_{\delta_x}\left[\sum_{u\in V_s} B(X_s^u)\sum_{k\geq 0} p_k(X_s^u)\left(\sum_{j=1}^k h_{n,{\gamma}}\left(F_j^{(k)}(X_s^u,\theta)\right)-h_{n,{\gamma}}\left(X_s^u\right)\right)\right]d\theta ds,
\end{align*}
where $h_{n,\gamma}$ is introduced in Assumption \ref{assu:debut2}. Letting $n$ tend to infinity and using Assumption \ref{assu:debut2} yields
%recalling that $h_{\gamma}(x)=\left(\sum_{i=1}^d |x_i|\right)^{\gamma}$, for $x\in\mathbb{R}^d$, we obtain:
\begin{align*}
\mathbb{E}_{\delta_x}\Bigg[\sum_{u\in V_t}&|X_t^u|^{\gamma}\Bigg]\leq |x|^{\gamma} +\int_0^t\mathbb{E}_{\delta_x}\left[\sum_{u\in V_s} \left(c_1\left|X_s^u\right|^{\gamma}+c_2\right)\right]ds\\
&+\int_0^t\int_{[0,1]}\mathbb{E}_{\delta_x}\left[\sum_{u\in V_s} B(X_s^u)\sum_{k\geq 0} p_k(X_s^u)\left(\sum_{j=1}^k\left|F_j^{(k)}(X_s^u,\theta)\right|^{\gamma}-\left|X_s^u\right|^{\gamma}\right)\right]d\theta ds.
\end{align*}
Next, using Assumption \ref{assu:debut1}(2) and \ref{assu:debut2}, we get that
 \begin{align*}
\mathbb{E}_{\delta_x}\left[\sum_{u\in V_t}|X_t^u|^{\gamma}\right]\leq  |x|^{\gamma} &+\int_0^t\mathbb{E}_{\delta_x}\left[ \sum_{u\in V_s} \left(c_1\left|X_s^u\right|^{\gamma}+c_2\right)\right]\\
& +\mathbb{E}_{\delta_x}\left[\sum_{u\in V_s} B(X_s^u)\ell(s)\mathbf{1}_{\left\{\left|X_s^u\right|^{\gamma}\leq \ell(s)\right\}}\right]ds.
 \end{align*}
Finally, using Assumption \ref{assu:debut1}(1) and the fact that $t\mapsto \ell(t)$ is increasing, we get
\begin{align*}
\mathbb{E}_{\delta_x}\left[\sum_{u\in V_t}|X_t^u|^{\gamma}\right]\leq |x|^{\gamma}&+c_1\int_0^t \E_{\delta_x}\left[\sum_{u\in V_s} \left|X_s^u\right|^{\gamma}\right]ds\\
&+\left(c_2+\left(b_1\ell(t)+b_2\right)\ell(t)\right)\int_0^t \E_{\delta_x}\left[N_s\right]ds.
\end{align*}
Combining this inequality with \eqref{eq:nbrmoyen}, we obtain for all $t\leq T_k\wedge T$
\begin{align*}
\E_{\delta_x}\left[N_t\right]+\E_{\delta_x}\left[\sum_{u\in V_t}|X_t^u|^{\gamma}\right]\leq 1+|x|^{\gamma}+A(T)\int_0^t \left(\E_{\delta_x}\left[N_s\right]+\E_{\delta_x}\left[\sum_{u\in V_s}|X_s^u|^{\gamma}\right]\right)ds,
\end{align*}
where $A(T)=c_1+c_2+b_1\ell(T)^2+b_2\ell(T)+(b_1+b_2)\overline{m}$.
 According to Gr\"onwall Lemma, for all $t\leq T_k\wedge T$, we get
\begin{align*}
\E_{\delta_x}\left[N_t\right]+\E_{\delta_x}\left[\sum_{u\in V_t}|X_t^u|^{\gamma}\right]\leq \left(1+|x|^{\gamma}\right) e^{A(T)t}<\infty.
\end{align*}

Finally, the average number of individuals in the population at time $t$ is bounded for $t$ in compact sets and there is no explosion of the population in finite time.
 \end{proof}
 
Before moving to the next section , we introduce $\left(R_{s,t},\ t\geq s\right)$, the first-moment semi-group associated with the branching process: for all $s\geq 0$, $t\geq s$ and $x\in\mathcal{X}$, let
\begin{align}\label{eq:firstmoment}
R_{s,t}f(x)=\E\Big[\sum_{u\in V_t} f\left(X_t^u\right)\Big|Z_s=\delta_x\Big],
\end{align}
where $f$ is a measurable function. Applying equation \eqref{eq:evol} to $f\equiv 1$, and taking the expectation yields
\begin{align}\label{eq:mtsx}
R_{s,t}\mathbf{1}(x)=m(x,s,t)=1+\int_s^t \mathbb{E}\Big[\sum_{u\in V_r}B(X_r^u)(m(X_r^u)-1)\Big| Z_s=\delta_x\Big]dr.
\end{align}
In particular, if $B\equiv b$ and $m(x)=m$ for all $x\in\mathcal{X}$, we obtain $m(x,s,t)=e^{b(m-1)(t-s)}$.

Finally, let us recall that for all $0\leq s\leq t$, $R_{s,t}$ is also a linear operator from the set of measures of finite mass into itself through the left action. In particular, for any $x\in \mathcal{X}$, we will denote the measure $\delta_x R_{s,t}(dy)$ by $R_{s,t}(x,dy)$.
\subsection{Some growth-fragmentation models for cell population dynamics}\label{sec:exdebut}

In this section, we consider growth-fragmentation processes: at division, the trait of the ancestor is shared between the children and the number of individuals in the population increases. Moreover, we focus on models where the trait moves according to a diffusion with associated generator of the form
\begin{align*}
\mathcal{G}f(x)=r(x)f'(x)+\sigma^2(x)f''(x),
\end{align*}
where $r$ and $\sigma$ are measurable functions. 
This class covers several dynamics for the trait. Here, we present three of them. In particular, we give an explicit formula for the average number of individuals in the population at time $t$.
We first give a useful equation concerning models with such a dynamic. For all $s\geq 0$, $t\geq s$ and $x\in\mathcal{X}$, applying \eqref{eq:evol} to $\mathbf{Id}(x)=x$ and taking the expectation, we obtain
\begin{align}\label{eq:evolalindentite}
R_{s,t}\mathbf{Id}(x)=x+\int_s^t R_{s,u}r(x)du,
\end{align}
where $(R_{s,t})_{t\geq s}$ is defined in \eqref{eq:firstmoment}.
\subsubsection{Linear growth model}\label{sub:TCP}
We consider here a size-structured model. More precisely, the size of each cell
grows linearly at a rate $a>0$ and this rate is supposed to be identical for each cell. We assume that divisions occur at rate $B(x)=\alpha x$, $\alpha>0$. At fission, the cell splits into two daughter cells of size
$\frac{x}{2}$, where $x$ denotes the size of the mother at splitting. Deciding whether the cells' growth follows a linear or an exponential dynamic has fueled a large debate in the literature (see \cite{cooper2006distinguishing} and references therein). The linear growth model has been considered for example in \cite{doumic2010} for the calibration of a deterministic growth-fragmentation model from experimental data and in \cite{hoang2015estimating} for the estimation of the division rate. 

Using the previous notations, the process $\left(X_t, t\geq 0\right)$ describing the size of a cell starting from $x_0$ is given by
\[
X_{t}=x_{0}+at,
\]
and the associated generator is given for any function $f\in\mathcal{C}^1(\mathbb{R}_+)$ by
\begin{align*}
\mathcal{G}f(x)=af'(x).
\end{align*}
Then, the branching process $(Z_t, t\geq 0)$ is solution of the following equation, for any function $f\in \mathcal{C}^1(\mathbb{R}_+)$ and any $x\in\mathcal{X}$,
\begin{align*}
\langle Z_t,f\rangle=&\langle Z_0,f\rangle +\int_0^t \int_{\mathbb{R}_+} af'(x)Z_s(dx)ds\\
&+\int_0^t\int_{\mathcal{U}\times \mathbb{R}_+}\mathbf{1}_{\left\{u\in V_{s^-},\ z\leq \alpha X_{s^-}^u\right\}}\left(2f\left(\frac{X_{s^-}^u}{2}\right)-f\left(X_{s^-}^u\right)\right)M(ds,du,dz),
\end{align*}
where $M$ is a Poisson point measure on $\mathbb{R}_+\times\mathcal{U}\times \mathbb{R}_+$ with intensity $ds\otimes n(du)\otimes dz$. The first integral corresponds to the dynamic of the population between two divisions. The integral with respect to the Poisson point measure represents to the jump part of the process and the indicator function corresponds to the fact that an individual $u$ jumps at time $s$ if it is in the population at time $s^-$ and if the division rate at $X_{s^-}^u$ is large enough. In this case, it is removed from the population and two descendants with trait $X_{s^-}^u/2$ appear. 

The validity of Assumptions \ref{assu:debut1} and \ref{assu:debut2} is trivial for this model with $\gamma = 1$. 
Let us compute the average number of individuals in the population at time $t$. For all $s\leq t$ and $x\in\mathbb{R}$, we have using \eqref{eq:mtsx}:
\begin{equation}\label{eq:mt}
m(x,s,t)=1+\alpha\int_s^t \mathbb{E}\left(\sum_{u\in V_{r}}X_{r}^{u}\middle| Z_s = \delta_x\right)dr.
\end{equation}
Combining \eqref{eq:evolalindentite} with $r(x)\equiv a$ and \eqref{eq:mt}, we obtain
\begin{align*}
m(x,s,t) =1+\alpha \int_{s}^{t}\left(x+a\int_{s}^{r}m(x,s,\tau)d\tau\right)dr,
\end{align*}
and for all $x\in\mathcal{X}$ and $s\geq 0$, $m(x,s,\cdot)$ is the solution of the following Cauchy problem with unknown $f$:
\[
\begin{cases}
f''(t)=a\alpha f(t),\\
f(s)=1,\ f'(s)=\alpha x.
\end{cases}
\]
with explicit solution given by
\[
m(x,s,t)=\frac{1}{2}\left(e^{\overline{a}(t-s)}+e^{-\overline{a}(t-s)}\right)+\frac{x}{2}\sqrt{\frac{\alpha}{a}}\left(e^{\overline{a}(t-s)}-e^{-\overline{a}(t-s)}\right),
\]
where $\overline{a}=\sqrt{a\alpha}$.
The population size is exponential in time as in the neutral case. 
\subsubsection{Exponential growth model in a varying environment }\label{sub:binary}

We assume here that the growth of the cells is exponential at rate $a$. This exponential growth model has been studied in \cite{doumic2015statistical} in the case of a specific growth rate for each individual in order to infer the division rate of the population. Here, we assume that the division rate is a function of time, mimicking a varying environment. More precisely, we set $B(x,t)=\alpha(t) x$, with $\alpha$ a positive function. The generator for the dynamic of the size is given for any function $f\in\mathcal{C}^1(\mathbb{R}_+)$ by
\begin{align*}
\mathcal{G}f(x)=axf'(x).
\end{align*}
We still assume that the branching is binary and that the size of the descendants at birth are both $x/2$ if $x$ is the size of the mother at splitting. Then, the branching process $(Z_t,t\geq 0)$ is solution of the following equation, for any function $f\in \mathcal{C}^1(\mathbb{R}_+)$ and any $x\in\mathcal{X}$: 
\begin{align*}
\langle Z_t,f\rangle=&\langle Z_0,f\rangle +\int_0^t \int_{\mathbb{R}_+} axf'(x)Z_s(dx)ds\\
&+\int_0^t\int_{\mathcal{U}\times \mathbb{R}_+}\mathbf{1}_{\left\{u\in V_{s^-},\ z\leq \alpha(s) X_{s^-}^u\right\}}\left(2f\left(\frac{X_{s^-}^u}{2}\right)-f\left(X_{s^-}^u\right)\right)M(ds,du,dz),
\end{align*}
where $M$ is a Poisson point measure on $\mathbb{R}_+\times\mathcal{U}\times \mathbb{R}_+$ with intensity $ds\otimes n(du)\otimes dz$.
Moreover, using \eqref{eq:evolalindentite} with $r(x)=ax$, we have
\[
\E\left(\sum_{u\in V_{t}}X_{t}^{u}\big| Z_s=\delta_x\right)=xe^{a(t-s)}.
\]
Combining this with equation \eqref{eq:mtsx}, we obtain
\[
m(x,s,t)=1+x\int_s^t \alpha(r)e^{a(r-s)}dr.
\]
In particular, if $\alpha(r)\equiv \alpha$ with $\alpha$ a positive constant, we get
\[
m(x,s,t)=1+\frac{\alpha x}{a}\left(e^{a(t-s)}-1\right).
\] 
The growth is again exponentially fast in time. 
\subsubsection{Parasite infection model}\label{sub:Kimmel}
This model is a continuous version of Kimmel's multilevel model for plasmids \cite{kimmel1997quasistationarity} which has already been studied in the case of a constant or monotone division rate by Bansaye and Tran in \cite{BT}. It models the proliferation of a parasite infection in a cell population. More precisely, we assume here that the trait $(X_t,t\geq 0)$ is a Markov process describing the quantity of parasites in each cell which evolves as a Feller diffusion process:
\begin{equation*}
X_t=X_0+\int_0^t gX_s ds+\int_0^t \sqrt{2\sigma^2 X_s}dB_s,
\end{equation*}
where $(B_s)_{s\geq 0}$ is standard Brownian motion and $g,\sigma>0$ are some fixed parameters. The generator for the dynamic of the quantity of parasites is given for any function $f\in\mathcal{C}^2(\mathbb{R}_+)$ by
\begin{align*}
\mathcal{G}f(x)=gxf'(x)+\sigma^2xf''(x).
\end{align*}
We assume here that a cell with a quantity $x$ of parasites will potentially divide at a rate $B(x)=\alpha x+\beta $, $\alpha,\beta >0$ into two daughter cells with a quantity $\delta x$ and $(1-\delta)x$ of parasites respectively, where $\delta$ is a random variable with uniform distribution on $[0,1]$. We need $\beta$ to be strictly positive so that even cells without any parasites divide after some time. The branching process $(Z_t,t\geq 0)$ is then solution of the following equation, for any function $f\in \mathcal{C}^2(\mathbb{R}_+)$ and any $x\in\mathcal{X}$: 
\begin{multline*}
\langle Z_t,f\rangle  =\langle Z_0,f\rangle +\int_0^t \int_{\mathbb{R}_+} \left(gxf'(x)+\sigma^2xf''(x)\right)Z_s(dx)ds+\mathcal{M}_t\\
+\int_0^t\int_{\mathcal{U}\times \mathbb{R}_+\times [0,1]}\mathbf{1}_{\left\{u\in V_{s^-},\ z\leq \alpha X_{s^-}^u+\beta\right\}}\left(f\left(\delta X_{s^-}^u\right)+f\left((1-\delta)X_{s^-}^u\right)-f\left(X_{s^-}^u\right)\right)\\
\times M(ds,du,dz,d\delta),
\end{multline*}
where $$\mathcal{M}_t = \int_0^t\sum_{u\in V_s} \sqrt{2\sigma^2 X_s^u}f'(X_s^u)dB^u_s$$
and $M$ is a Poisson point measure on $\mathbb{R}_+\times\mathcal{U}\times \mathbb{R}_+\times [0,1]$ with intensity $ds\otimes n(du)\otimes dz\otimes d\delta$ and $(B_s^u,s\geq 0)_{u\in\mathcal{U}}$ is a family of independent standard Brownian motions. In particular, the generator corresponding to first moment semi-group is given for any function $f\in \mathcal{C}^2(\mathbb{R})$ and $x\in\mathcal{X}$ by
\begin{equation*}
\mathcal{F}_{\text{inf}}f(x)=gxf'(x)+\sigma^2xf''(x)+\left(\alpha x+\beta\right)\left(\int_0^1\left[f\left(\delta x\right)+f\left((1-\delta)x\right)\right]d\delta-f(x)\right).
\end{equation*}
Therefore, we notice that if $(V,\lambda)$ are eigenelements of $\mathcal{F}_{\text{inf}}$, we have $\mathcal{F}_{\text{inf}}V(0)=\beta V(0)$ so that $V(0)=0$ if $\lambda\neq \beta$ and we cannot apply usual techniques using eigenelements requiring that $V> 0$ \cite{cloez}.

Let us compute the average number of individuals in the population after time $t$. Using \eqref{eq:mtsx}, we have
\[
m(x,s,t)=1+\alpha\int_s^t \E\Big[\sum_{u\in V_r} X_r^u\Big| Z_s=\delta_x\Big]dr+\beta\int_s^t m(x,s,r)dr. 
\]
Again, using \eqref{eq:evolalindentite}, we obtain
\[
\E\Big[\sum_{u\in V_r} X_r^u\Big| Z_s=\delta_x\Big]=xe^{g(r-s)}.
\]
Then, combining the two previous equations and differentiating, we get
\begin{align*}
\partial_t m(x,s,t) = \alpha xe^{g(t-s)}+\beta m(x,s,t),
\end{align*}
and finally
\[
m(x,s,t)=\frac{\alpha x}{g-\beta}e^{g(t-s)}+\left(1-\frac{\alpha x}{g-\beta}\right)e^{\beta(t-s)},
\]
if $g\neq \beta$ and:
\[
m(x,s,t)=\left(1+\alpha x(t-s)\right)e^{\beta(t-s)},
\]
if $g=\beta$.
In the three examples above, the mean number of individuals in the population is an affine function of the trait of the initial individual. However, this is not the rule. For example, Cloez developed in \cite[Corollary 6.1]{cloez} the case of a dynamic of the trait following an Ornstein-Uhlenbeck process where the dependence in $x$ is not affine. 

For other examples and comments, including a link with the integro-differential model, we refer to Section \ref{sec:other}.

\section{The trait of sampled individuals at a fixed time : Many-to-One formulas\label{sampling}}
In order to characterize the trait of a uniformly sampled individual, the spinal approach (\cite{chauvin1988kpp},\cite{lyons1995conceptual}), consists in following a "typical" individual in the population whose behavior summarizes the behavior of the entire population. Biggins \cite{biggins1977martingale} used this approach for the study of branching random walks extending Kingman results \cite{kingman}. The spinal approach has then been extended to various frameworks (\cite{harris1996large},\cite{kurtz1997conceptual},\cite{hardy2009spine}). In particular, Georgii and Baake \cite{georgii2003supercritical} used spine techniques in a spectral framework to describe the asymptotic distribution of the trait of a uniformly sampled individual in the population and its ancestral lineage in the case of a finite set of possible trait.

In this section, we specify the generator of the process describing the trait along the spine. The existence of our auxiliary process does not rely on the existence of spectral elements for the mean operator of the branching process. 

With a slight abuse of notation, for all $u\in V_t$ and $s<t$, we denote by $X_s^u$ the trait of the unique ancestor living at time $s$ of $u$. 

\subsection{The auxiliary process}
Let us define
\begin{align*}
\mathcal{D}(\mathcal{A})=\left\lbrace f\in\mathcal{D}(\mathcal{G})\text{ s.t. } m(\cdot,s,t)f(s,x)\in\mathcal{D}(\mathcal{G})\ \forall t\geq 0, s\leq t\right\rbrace.
\end{align*}
From now on, we assume that for all $x\in\mathcal{X}$, $t\geq 0$ and $s\leq t$, $m(x,s,t)\neq 0$.

We now recall the operator and functions needed for the definition of the auxiliary process, and introduce additional notations. For all $f\in\mathcal{D}(\mathcal{A})$, $x\in\mathcal{X}$ and $s<t$, we write
\begin{align}\label{eq:genemov_auxi}
\widehat{\mathcal{G}}_{s}^{(t)}f(x)=\frac{\mathcal{G}\left(m(\cdot,s,t)f\right)(x)-f\left(x\right)\mathcal{G}\left(m(\cdot,s,t)\right)(x)}{m(x,s,t)},
\end{align}
\begin{align}\label{eq:taux_auxi}
\widehat{B}_{s}^{(t)}(x)=B(x)\Lambda(x,s,t),
\end{align}
\begin{align}\label{eq:noyaudiv_auxi}
\widehat{P}_{s}^{(t)}(x,dy)=\Lambda^{-1}(x,s,t)\frac{m(y,s,t)}{m(x,s,t)}m(x,dy),
\end{align}
where
\[
\Lambda(x,s,t)=\int_{\mathcal{X}}\frac{m(y,s,t)}{m(x,s,t)}m(x,dy).
\]

 In order to prove a Many-to-One formula, we need to consider the following assumptions:
\begin{assumptionA}\label{assu:doeblin}
There exists a function $C$ such that for all $j\leq k,\ j,k\in\mathbb{N}$ and $0\leq s\leq t$, we have
\[
\sup_{x\in\mathcal{X}}\sup_{s\in[0,t]}\int_\mathcal{X}\frac{m(y,s,t)}{m(x,s,t)}P_j^{(k)}(x,dy)\leq C(t),\ \forall t\geq 0.
\]
\end{assumptionA}

\begin{assumptionA}\label{assu:differentiable}
For all $t\geq 0$, we have
\begin{itemize}
\item[-] for all $x\in\mathcal{X}$, $s\mapsto m(x,s,t)$ is differentiable on $[0,t]$ and its derivative is continuous on $[0,t]$,
\item[-] for all $x\in\mathcal{X}$, $f\in\mathcal{D}(\mathcal{A})$, $s\mapsto \mathcal{G}(m(\cdot,s,t)f)(x)$ is continuous,% pour que le DL s'écrive avec les dérivées partielles
\item[-]$\mathcal{D}(\mathcal{A})$ is dense in $\mathcal{C}_b(\mathcal{X})$ for the topology of uniform convergence.
\end{itemize}

\end{assumptionA}

The last item of this assumption allows us to extend our formulas to all measurable functions with respect to the Skorokod topology using a monotone class argument. Moreover, combining Lemma \ref{lemma:mdansG} and Remark \ref{rem:densite}, this assumption is in particular satisfied if $\mathcal{D}(\mathcal{G})$ is stable by product.
\begin{thm}\label{th:mto}
Under Assumptions \ref{assu:debut1}(1-3), \ref{assu:debut2}, \ref{assu:doeblin} and \ref{assu:differentiable}, for all $t\geq 0$, for all $x_0\in\mathcal{X}$ and for all non-negative measurable functions $F:\mathbb{D}\left([0,t],\mathcal{X}\right)\rightarrow\mathbb{R}_+$, we have
\begin{equation}\label{mtomesurable}
\E_{\delta_{x_0}}\left[\sum_{u\in V_{t}}F\left(X_{s}^{u},s\leq t\right)\right]=m(x_0,0,t)\E_{x_0}\left[F\left(Y_{s}^{(t)},s\leq t\right)\right],
\end{equation}
where $\left(Y_{s}^{(t)}, s\leq t\right)$ is a time-inhomogeneous Markov process whose law is characterized by its associated infinitesimal generators $\left(\mathcal{A}_{s}^{(t)}\right)_{s\leq t}$ given for $f\in\mathcal{D}(\mathcal{A})$ and $x\in\mathcal{X}$ by
\begin{align}
\mathcal{A}_{s}^{(t)}f(x)= & \widehat{\mathcal{G}}_{s}^{(t)}f(x)
  +\widehat{B}_{s}^{(t)}(x)\int_{\mathcal{\mathcal{X}}}\left(f\left(y\right)-f\left(x\right)\right)\widehat{P}_{s}^{(t)}\left(x,dy\right).\label{eq:gene_auxi}
\end{align}

\end{thm}
Formula \eqref{mtomesurable} has a natural interpretation in terms of semi-groups. If $f$ is a non-negative measurable function, for any $0\leq r\leq s\leq t$ and any $x\in\mathcal{X}$, we set 
\begin{align}\label{eq:semigroup_auxi}
P_{r,s}^{(t)}f(x):=\frac{\mathbb{E}\left[\sum_{u\in V_t} f\left(X_s^u\right)\Big| Z_r=\delta_x\right]}{m(x,r,t)}=\mathbb{E}\left[f\left(Y_s^{(t)}\right)\Big|Y_r^{(t)}=x\right].
\end{align}
In other words, $\left(P_{r,s}^{(t)}, r\leq s\leq t\right)$ is a conservative time-inhomogeneous semi-group i.e. for all $r\leq u\leq s\leq t$, $P_{r,u}^{(t)}P_{u,s}^{(t)}=P_{r,s}^{(t)}$, and the auxiliary process $Y^{(t)}$ is its time-inhomogeneous associated Markov process corresponding to the right-hand side of \eqref{eq:mto_poid}. We can exhibit this process using a change of probability measure. Indeed, by Feynman-Kac's formula \cite[Section 1.3]{del2004feynman}, we have
\begin{align*}
P_{r,s}^{(t)}f(x)=m(x,r,t)^{-1}\mathbb{E}\left[e^{\int_r^sB(X_v)(m(X_v)-1)dv}m(X_s,s,t)f(X_s)\Big|X_r=x\right],
\end{align*}
where the Markov process $(X_s,r\leq s\leq t)$ corresponds to dynamic of the tagged-particle which infinitesimal generator $\mathcal{M}$ is given by
\begin{align*}
\mathcal{M}f(x)=\mathcal{G}f(x)+B(x)m(x)\sum_{k\geq 0}\frac{kp_k(x)}{m(x)}\frac{1}{k}\sum_{i=1}^k\int_{\mathcal{X}}\left(f(y)-f(x)\right)P_i^{(k)}(x,dy).
\end{align*}
Then, the change of probability measure given by the $\sigma(X_l,l\leq s)$-martingale
\begin{align*}
M_s^{(t)}:=\frac{e^{\int_r^sB(X_s)(m(X_s)-1)ds}m(X_s,s,t)}{m(x,r,t)},\quad \text{for }r\leq s\leq t
\end{align*}
exhibits the probability measure corresponding to the auxiliary process.

Before proving Theorem \ref{th:mto}, we give some links between our approach and previous works on this subject. 
First, in the neutral case, i.e. $B$ and $(p_{k})_{k\in\mathbb{N}}$ constants, the auxiliary process coincides with the one in \cite{bansaye2011limit} i.e. for all $f\in \mathcal{D}(\mathcal{G})$ and $x\in\mathcal{X}$, the infinitesimal generator of the auxiliary process is given by
\begin{align*}
\mathcal{A}f(x)=\mathcal{G}f(x)+Bm\sum_{k\geq 0}\widehat{p}_k\left(\frac{1}{k}\sum_{j=1}^k \int_{\mathcal{X}}\left(f(y)-f(x)\right)P_j^{(k)}(x,dy)\right),
\end{align*}
where $\widehat{p}_k=kp_km^{-1}$ denote the biased reproduction law.
In the general case, the dynamic of the auxiliary process heavily depends on the comparison between $m(x,s,t)$ and  $m(y,s,t)$, for $x,y\in\mathcal{X}$. It emphasizes several bias due to growth of the population. First, the auxiliary process jumps more than the original process, if jumping is beneficial in terms of number of descendants. This phenomenon of time-acceleration also appears for examples in \cite{chauvin1988kpp}, \cite{lyons1995conceptual} or \cite{hardy2009spine}. Moreover, the reproduction law favors the creation of a large number of descendants as in \cite{bansaye2011limit} and the non-neutrality favors individuals with an "efficient" trait at birth in terms of number of descendants. Finally, a new bias appears on the dynamic of the trait because of the combination of the random evolution of the trait and non-neutrality. Indeed, if the dynamic of the trait is deterministic, we have $\widehat{\mathcal{G}}_s^{(t)}f(x)=\mathcal{G}f(x)$.

The auxiliary process could be guessed through a discretization of the model using the expression of the
auxiliary process in \cite[eq. 1]{bansaye2013}. However, the proof of Theorem \ref{th:mto} is a direct continuous
time approach relying on the uniqueness of
%The auxiliary process has been guessed through a discretization of the model using the expression of the auxiliary process in \cite{bansaye2013}. More precisely, we considered the discrete time auxiliary process with a probability of branching at each step depending on the division rate $B$. Then, we looked at the limit of the generator of this process as the discretization step goes to zero in order to obtain a formula for the continuous-time version of the auxiliary process. However, the proof of Theorem \ref{th:mto} does not rely on a discretization argument but on the uniqueness of 
the solution to the integro-differential equation \eqref{eq:mu}. The proof is decomposed in four parts: first, in Lemma \ref{prop:L'ensemble-des-mesures}, we prove that the integro-differential equation \eqref{eq:mu} admits a unique solution which corresponds to the semi-group of the auxiliary process defined in \eqref{eq:semigroup_auxi}. Afterwards, in Lemma \ref{lemma:autreauxi}, we prove that the infinitesimal generator of this auxiliary process verifies \eqref{eq:gene_auxi}. Then, we prove Theorem \ref{th:mto} for any function such that $F(x)=f_1(x_{t_1})\ldots f_k(x_{t_k})$, $x\in\mathbb	{D}([0,t],\mathcal{X})$, by induction on $k\in\mathbb{N}$. Finally, we extend the set of functions for which \eqref{mtomesurable} is satisfied using a monotone class argument.

Let $t\geq 0$. We define the following family of semi-groups for $f\in\mathcal{D}(\mathcal{A})$:
\begin{align*}
Q_{s,r}^{(t)}f(x)=\frac{A_{r-s}(m(\cdot,r,t)f)(x)}{m(x,s,t)},\ s\leq r\leq t.
\end{align*}
We also define
\begin{align}\label{eq:genemarkov}
\widetilde{\mathcal{G}}_s^{(t)}f(x)=\frac{\mathcal{G}(m(\cdot,s,t)f)(x)+f(x)\partial_sm(x,s,t)}{m(x,s,t)}.
\end{align}
\begin{lemma}
\label{prop:L'ensemble-des-mesures}Let $t\geq 0$. Under Assumptions \ref{assu:debut1}(1-3), \ref{assu:debut2}, \ref{assu:doeblin} and \ref{assu:differentiable}, for all $x_0\in\mathcal{X}$ and $t_0\leq t$, the family of probability measures $\left(P_{t_0,s}^{(t)}(x_0,\cdot),t_0\leq s\leq t\right)$ is the unique solution of the following equation with unknown $\left(\mu_{t_0,s}(x_0,\cdot),\ t_0\leq s\leq t\right)$:
\begin{align}\nonumber 
\mu_{t_0,s}\left(x_{0},f\right)= & f\left(t_0,x_{0}\right)+\int_{t_0}^{s}\int_{\mathcal{X}}\left(\widetilde{\mathcal{G}}_{r}^{(t)}f(r,x)+\partial_{r}f(r,x)\right)\mu_{t_0,r}\left(x_{0},dx\right)dr\label{eq:mu}\\
 & +\int_{t_0}^{s}\int_{\mathcal{\mathcal{X}}}\left[\widehat{B}_{r}^{(t)}(x)\int_{\mathcal{X}}f\left(r,y\right)\widehat{P}_{r}^{(t)}(x,dy)-B(x)f\left(r,x\right)\right]\mu_{t_0,r}\left(x_{0},dx\right)dr,
\end{align} 
for all function $f\in \mathcal{D}(\mathcal{A})$ such that $s\mapsto f(s,x)$ is continuously differentiable for all $x\in\mathcal{X}$.
\end{lemma}
\begin{proof} Let $t\geq 0$ and let $f$ be as in the statement of the lemma. The proof falls naturally into two parts. We begin by proving that $\left(P_{t_0,s}^{(t)}(x_0,\cdot),t_0\leq s\leq t\right)$ is a solution of \eqref{eq:mu}. First, we notice that for all $t_0\leq s\leq t,\ x_0\in\mathcal{X}$, \[m(x_0,t_0,t)P_{t_0,s}^{(t)}f(t_0,x_0)=\E\left(\langle Z_{s},f(s,\cdot)m(\cdot,s,t)\rangle \big|Z_{t_0}=\delta_{x_0}\right).\]
Indeed, from \eqref{eq:semigroup_auxi}, we have
\begin{align*}
m(x_0,t_0,t)P_{t_0,s}^{(t)}f(t_0,x_0) & =\E\Big[\sum_{v\in V_{s}}\sum_{
u\in V_{t},u\geq v}
f\left(s,X_{s}^{v}\right)\Big|Z_{t_0}=\delta_{x_0}\Big]\\
&=\E\Big[\sum_{v\in V_{s}}f\left(s,X_{s}^{v}\right)\E\Big(\sum_{u\in V_{t},u\geq v}1\Big|\mathcal{F}_{s}\Big)\Big|Z_{t_0}=\delta_{x_0}\Big]\\
 & =\E\Big[\sum_{v\in V_{s}}f\left(s,X_{s}^{v}\right)\E\Big(\sum_{\begin{subarray}{c}
u\in V_{t}\end{subarray}}1\Big|Z_s=\delta_{X_s^v}\Big)\Big|Z_{t_0}=\delta_{x_0}\Big]\\
& =\E\Big[\sum_{v\in V_{s}}f\left(s,X_{s}^{v}\right)m\left(X_{s}^{v},s,t\right)\Big|Z_{t_0}=\delta_{x_0}\Big].
\end{align*}
Then, applying \eqref{eq:evol} to the function $(x,s)\mapsto f(s,x)m(x,s,t)$ and taking the expectation, we obtain
\begin{multline}\label{eq:nu}
\E\left[\sum_{v\in V_{s}}f\left(s,X_{s}^{v}\right)m\left(X_{s}^{v},s,t\right)\Big|Z_{t_0}=\delta_{x_0}\right]= m(x_0,t_0,t)f\left(t_0,x_{0}\right)\\
+\int_{t_0}^{s}\int_{\mathcal{X}}\left(\mathcal{G}\left(f(r,\cdot)m(\cdot,r,t)\right)\left(x\right)+f(r,x)\partial_{r}m(x,r,t)+\partial_r f(r,x)m(x,r,t)\right)R_{t_0,r}(x_0,dx)dr\\
  +\int_{t_0}^{s}\int_{\mathcal{X}}B(x)\left(\sum_{k\geq0}p_{k}(x)\sum_{j=1}^{k}\int_{\mathcal{X}}f\left(r,y\right)m\left(y,r,t\right)P_{j}^{(k)}\left(x,dy\right)-f\left(r,x\right)m\left(x,r,t\right)\right)\\
  \times R_{t_0,r}(x_0,dx)dr.
\end{multline}
Finally, factorizing by $m(x,r,t)$ in the last two terms and dividing by $m(x_0,t_0,t)$, we obtain that the family of probability measures $\left(P_{t_0,s}^{(t)}(x_0,\cdot),t_0\leq s\leq t\right)$ is a solution of \eqref{eq:mu}.

We now prove the uniqueness of a solution to \eqref{eq:mu}.  Without loss of generality, we assume that $t_0=0$.
This part of the proof is adapted from \cite{BT}. Let $\left(\gamma_{s,t}^{1}, s\leq t\right)$ and $\left(\gamma_{s,t}^{2},s\leq t\right)$ be two solutions of equation \eqref{eq:mu}. Let us recall that the total variation norm is given for all measures $\gamma^1,\gamma^2$ on $\mathcal{X}$ with finite mass by
\[
\left\Vert \gamma^{1}-\gamma^{2}\right\Vert _{TV}=\sup_{\phi\in\mathcal{C}_{b}\left(\mathcal{X},\mathbb{R}\right),\left\Vert \phi\right\Vert _{\infty}\leq1}\left|\gamma^{1}(\phi)-\gamma^{2}\left(\phi\right)\right|,
\]
where $\mathcal{C}_{b}\left(\mathcal{X},\mathbb{R}\right)$ denotes the set of continuous bounded functions from $\mathcal{X}$ to $\mathbb{R}$.
The idea is to find a function which cancels the first integral in \eqref{eq:mu}.
Let $x\in\mathcal{X}$, $t\geq 0$ and $r\leq t$.
We begin by computing the differential of $\left(Q_{s,r}^{(t)}f(x),\ s\leq r\leq t\right)$ with respect to $s$. First, $s\mapsto A_{r-s}(m(\cdot,r,t)f)(x)$ is differentiable because $x\mapsto m(x,r,t)f(x)\in \mathcal{D}(\mathcal{G})$ and according to the backward equation, its derivative is $s\mapsto \mathcal{G}(A_{r-s}(m(\cdot,r,t)f))(x)=A_{r-s}(\mathcal{G}(m(\cdot,r,t)f))(x)$. Furthermore, $s\mapsto m(x,s,t)^{-1}$ is differentiable because $s\mapsto m(x,s,t)$ is differentiable according to the first point of Assumption \ref{assu:differentiable} and because $m(x,s,t)\neq 0$ for all $x\in\mathcal{X}$, $t\geq 0$ and $s\leq t$.
Then, for all $s\geq 0$ and $r\geq s$, we have
\begin{align*}
\partial_sQ_{s,r}^{(t)}f(x)&=\frac{\partial_sA_{r-s}\left(m(\cdot,r,t)f\right)(x)}{m(x,s,t)}-\frac{\partial_sm(x,s,t)}{m(x,s,t)^2}A_{r-s}\left(m(\cdot,r,t)f\right)(x)\\
&=-\frac{\mathcal{G}\left(A_{r-s}\left(m(\cdot,r,t)f\right)\right)(x)}{m(x,s,t)}-\frac{\partial_sm(x,s,t)}{m(x,s,t)}\frac{A_{r-s}\left(m(\cdot,r,t)f\right)(x)}{m(x,s,t)}\\
&=-\left(\frac{\mathcal{G}\left(m(\cdot,r,t)Q_{s,r}^{(t)}f\right)(x)}{m(x,s,t)}+\frac{\partial_sm(x,s,t)}{m(x,s,t)}Q_{s,r}^{(t)}f(x)\right).
\end{align*}
Therefore, for all $s\leq t$ and $f\in \mathcal{D}\left(\mathcal{A}\right)$, we have
\begin{equation}
\partial_s Q_{s,r}^{(t)}f(x)=-\widetilde{\mathcal{G}}_s^{(t)}Q_{s,r}^{(t)}f(x).\label{eq:5}
\end{equation}
Let $f\in \mathcal{D}\left(\mathcal{A}\right)$ be such that $\left\Vert f\right\Vert _{\infty}\leq1$. 
Let us consider $\tau_n(x)=\inf\left\lbrace t\geq 0, X_t\notin \mathcal{B}(x,n)\right\rbrace$, where $\mathcal{B}(x,n)=\left\lbrace y\in\mathcal{X},\ \left| x-y\right|\leq n\right\rbrace$. For all $x\in\mathcal{X}$, $s\leq r\leq t$ and $n\in\mathbb{N}$, we define
\begin{align*}
Q_{s,r}^{(t),n}f(x)=\frac{\mathbb{E}_x\left[m(X_{r\wedge\tau_n(x)-s},r\wedge\tau_n(x),t)f(X_{r\wedge\tau_n(x)-s})\right]}{m(x,s,t)}.
\end{align*}
We still have $\partial_s Q_{s,r}^{(t),n}f(x)=-\widetilde{\mathcal{G}}_s^{(t)}Q_{s,r}^{(t),n}f(x)$. Moreover, for all $s\leq r\leq t$ and all $x\in\mathcal{X}$, we have
\begin{align*}
\left|Q_{s,r}^{(t),n}f(x)\right|& \leq \frac{\mathbb{E}_x\left[m(X_{r\wedge \tau_n(x)-s},r\wedge \tau_n(x),t)\right]}{m(x,s,t)}\\
& \leq \frac{\mathbb{E}_x\left[m(X_{r\wedge \tau_n(x)-s},r\wedge \tau_n(x),t)\right]}{\E\left[\mathbf{1}_{\Omega_{r\wedge \tau_n(x)}}m(X_{r\wedge \tau_n(x)}^{\emptyset},r\wedge \tau_n(x),t)\big|Z_s=\delta_x\right]},
\end{align*}
where $\Omega_r=\left\lbrace T_1(\bar{Z})> r \right\rbrace$. Conditioning with respect to $\sigma\left(X_s^{\emptyset}, s\leq r\wedge \tau_n(x)\right)$ on the denominator, we obtain
\begin{align}\nonumber
\left|Q_{s,r}^{(t),n}f(x)\right|&\leq\frac{\mathbb{E}_x\left[m(X_{r\wedge \tau_n(x)-s},r\wedge \tau_n(x),t)\right]}{\E\left[\exp\left(-\int_0^{r\wedge\tau_n(x)} B(X_u^{\emptyset})du\right)m(X_{r\wedge \tau_n(x)}^{\emptyset},r\wedge \tau_n(x),t)\Big|X^{\emptyset}_s=x\right]}\\\label{eq:majoration}
&\leq \frac{\mathbb{E}\left[m(X_{r\wedge \tau_n(x)},r\wedge \tau_n(x),t)\big|X_s=x\right]}{\exp\left(-r \overline{B}_n(x)\right)\E\left[m(X_{r\wedge \tau_n(x)}^{\emptyset},r\wedge \tau_n(x),t)\right|\left.X^{\emptyset}_s=x\right]}\leq e^{r\overline{B}_n(x)},
\end{align}
where $\overline{B}_n(x)=\sup_{y\in\mathcal{B}(x,n)}B\left(y\right)$.

Let $T_{n}=\inf\left\{ s\leq t,\gamma^1_{s,t}\left(x_0,\mathcal{B}(x_0,n)^C\right)+\gamma^2_{s,t}\left(x_0,\mathcal{B}(x_0,n)^C\right)> 0\right\} $ where $\mathcal{B}(x_0,n)^C$ is the complementary of $\mathcal{B}(x_0,n)$ with the convention that $\inf \emptyset =+\infty$. Then, using that $\left(\gamma_{s,t}^i,s\leq t\right)$, for $i=1,2$, are solutions of \eqref{eq:mu}, we have for all $s\leq r\leq t$
\begin{align*}
&\langle \gamma_{s\wedge T_{n},t}^i(x_0,\cdot),Q_{s\wedge T_{n},r}^{(t),n}f\rangle= Q_{0,r}^{(t),n}f(x_0)\\
&+\int_{0}^{s\wedge T_{n}}\int_{\mathcal{\mathcal{X}}}\left[\widehat{B}_{u}^{(t)}\left(x\right)\int_{\mathcal{X}}Q_{u,r}^{(t),n}f(y)\widehat{P}_{u}^{(t)}\left(x,dy\right)-B(x)Q_{u,r}^{(t),n}f(x)\right]\gamma_{u,t}^i\left(x_0,dx\right)du.
\end{align*}
Using \eqref{eq:majoration}, we get
\begin{align*}
 & \left|\gamma_{s\wedge T_{n},t}^{1}\left(x_0,Q_{s\wedge T_{n},r}^{(t),n}f\right)-\gamma_{s\wedge T_{n},t}^{2}\left(x_0,Q_{s\wedge T_{n},r}^{(t),n}f\right)\right|\\
  \begin{split}
 & =\Bigg|\int_{0}^{s\wedge T_{n}}\int_{\mathcal{X}}\left[\widehat{B}_{u}^{(t)}\left(x\right)\int_{\mathcal{X}}Q_{u,r}^{(t),n}f\left(y\right)\widehat{P}_{u}^{(t)}\left(x,dy\right)-B(x)Q_{u,r}^{(t),n}f\left(x\right)\right]\\
&\hspace{7cm} \times \left(\gamma_{u,t}^{1}-\gamma_{u,t}^{2}\right)\left(x_0,dx\right)du\Bigg|
\end{split}\\
\begin{split}
& =\Bigg|\int_{0}^{s\wedge T_{n}}\int_{\mathcal{X}}B(x)\left[\int_{\mathcal{X}}Q_{u,r}^{(t),n}f\left(y\right)\frac{m(y,u,t)}{m(x,u,t)}m(x,dy)-Q_{u,r}^{(t),n}f\left(x\right)\right]\\
&\hspace{7cm} \times\left(\gamma_{u,t}^{1}-\gamma_{u,t}^{2}\right)\left(x_0,dx\right)du\Bigg|
 \end{split}\\
&  \leq (C(t)\overline{m}+1)e^{r\overline{B}_{r(n,x_0)}(x_0)}\overline{B}_{n}(x_0)\int_{0}^{s\wedge T_{n}}\left\Vert \gamma_{u,t}^{1}-\gamma_{u,t}^{2}\right\Vert _{TV}du\\
&  \leq (C(t)\overline{m}+1)e^{r\overline{B}_{r(n,x_0)}(x_0)}\overline{B}_{n}(x_0)\int_{0}^{s}\left\Vert \gamma_{u\wedge T_{n},t}^{1}-\gamma_{u\wedge T_{n},t}^{2}\right\Vert _{TV}du,
\end{align*}
where $r(n,x_0)=4n+2|x_0|+\ell(s)$ and $C(t)$ is defined in Assumption \ref{assu:doeblin}. Then Gr\"{o}nwall's lemma implies that $\left\Vert \gamma_{s\wedge T_{n},t}^{1}-\gamma_{s\wedge T_{n},t}^{2}\right\Vert _{TV}=0.$
Taking the limit as $n$ tends to $+\infty$, we obtain $\left\Vert \gamma_{s,t}^{1}-\gamma_{s,t}^{2}\right\Vert _{TV}=0$ and the uniqueness
of the solution to \eqref{eq:mu}. 
\end{proof}
\begin{lemma}\label{lemma:autreauxi}
Let $t\geq 0$. Under Assumption \ref{assu:differentiable}, the generator of $\left(P_{r,s}^{(t)},\ r\leq s\leq t\right)$, the semi-group defined in \eqref{eq:semigroup_auxi}, is $\left(\mathcal{A}_s^{(t)},\ s\leq t\right)$ defined on $\mathcal{D}(\mathcal{A})$.
\end{lemma}
For the proof of this Lemma, we need a preliminary result which proof is given in Section \ref{app:mdansG} in the appendix.
\begin{lemma}\label{lemma:mdansG}
For all $t\geq 0$ and $s\leq t$, 
\begin{align*}
\mathcal{G}(m(\cdot,s,t))(x) = & \lim_{r\rightarrow 0}\frac{\mathbb{E}(m(X_r,s,t)\big|X_0=x)-m(x,s,t)}{r}\\
 = & -\partial_{s}m(x,s,t)+B\left(x\right)m(x,s,t)\\
& -B\left(x\right)\sum_{k\geq0}p_{k}\left(x\right)\sum_{j=1}^{k}\int_{\mathcal{X}}m(y,s,t)P_{j}^{(k)}\left(x,dy\right).
\end{align*}
\end{lemma}
We can now prove Lemma \ref{lemma:autreauxi}.
\begin{proof}Let $t\geq 0$ and $f\in\mathcal{D}(\mathcal{A})$. If we take the expectation of \eqref{eq:evol} and differentiate with respect to $t$, for all function $g$ such that $g(s,\cdot)\in \mathcal{D}(\mathcal{A})$, we get that
\begin{align*}
\partial_t R_{s,t}g(x,s)=\mathcal{G}g(x,s)+\partial_s g(x,s)+B(x)\left(\int_{\mathcal{X}} g(y,s)m(x,dy)-g(x,s)\right):=\mathcal{R}g(x,s),
\end{align*}
for all $x\in\mathcal{X}$ and $s\leq t$, because $t\mapsto \mathbb{E}\left[\langle Z_t,f\rangle\right]$ is continuous whenever $f$ is continuous. 
Next, according to Assumption \ref{assu:differentiable}, we have the following first order Taylor expansion: for all $x\in\mathcal{X}$, $r<t$ and $h>0$,
\begin{align*}
P_{r,r+h}^{(t)}f(x)&= \frac{R_{r,r+h}(m(\cdot,r+h,t)f)(x)}{m(x,r,t)}\\
&=f(x)+\frac{\mathcal{R}(m(\cdot,r,t)f)(x)}{m(x,r,t)}h+\frac{\partial_r m(x,r,t)f(x)}{m(x,r,t)}h+o(h).
\end{align*}
Then,
\begin{align*}
\lim_{h\rightarrow 0} \frac{P_{r,r+h}^{(t)}f(x)-f(x)}{h}=\frac{\mathcal{R}(m(\cdot,r,t)f)(x)}{m(x,r,t)}+\frac{\partial_r m(x,r,t)f(x)}{m(x,r,t)},
\end{align*}
and we get
\begin{align}\label{eq:auxi1}
\mathcal{A}_r^{(t)}f(x)= \widetilde{\mathcal{G}}_{r}^{(t)}f(x)+\widehat{B}_{r}^{(t)}(x)\left[\int_{\mathcal{X}}f\left(y\right)\widehat{P}_{r}^{(t)}(x,dy)-B(x)f\left(x\right)\right],\quad r\leq t.
\end{align}
%However, taking $f\equiv 1$ in \eqref{eq:nu}, which is possible according to Lemma \ref{lemma:mdansG}, and differentiating with respect to $s$ yields:
%\begin{align*}
%&\mathbb{E}\left[\sum_{u\in V_s}\left(\mathcal{G}\left(m(\cdot,s,t)\right)\left(X_s^u\right)+\partial_{s}m(x,s,t)\right)\big|Z_{t_0}=\delta_{x_0}\right]\\
%&=-\mathbb{E}\left[B\left(X_s^u \right)\left(\sum_{k\geq0}p_{k}\left(X_s^u \right)\sum_{j=1}^{k}\int_{\mathcal{X}}m\left(y,s,t\right)P_{j}^{(k)}\left(X_s^u,dy\right)-m\left(X_s^u ,s,t\right)\right)\big|Z_{t_0}=\delta_{x_0}\right].
%\end{align*}
%Then, for $s = t_0$ we get : 
%\begin{align}\label{eq:genenul}
%\partial_{s}m(x,t_0,t)=-\mathcal{G}\left(m(\cdot,t_0,t)\right)\left(x\right)+B\left(x\right)m(x,t_0,t)-B\left(x\right)\sum_{k\geq0}p_{k}\left(x\right)\sum_{j=1}^{k}\int_{\mathcal{X}}m(y,t_0,t)P_{j}^{(k)}\left(x,dy\right),
%\end{align}
%for all $t_0\geq 0$. 
Combining Lemma \ref{lemma:mdansG} and \eqref{eq:auxi1}, we obtain formula \eqref{eq:gene_auxi} for the generator of the auxiliary process.
\end{proof}
\begin{proof}[Proof of Theorem \ref{th:mto}]
We prove the result by induction on $k\in\mathbb{N}$ for any separable function $F=f_1\ldots f_k$ with $f_i\in \mathcal{D}(\mathcal{A})$ for all $i=1\ldots k$.
We consider the following proposition denoted by $\mathcal{H}_{k}$: for all $0<s_{1}\leq s_{2}\leq\ldots  \leq s_{k}\leq t$, for all $x_0\in\mathcal{X}$ and $f_1,\ldots,f_k\in\mathcal{D}(\mathcal{A})$,
\begin{align*}
\E_{\delta_{x_0}}\left[\sum_{u\in V_{t}}f_{1}\left(X_{s_{1}}^{u}\right)\ldots  f_{k}\left(X_{s_{k}}^{u}\right)\right]=m\left(x_0,0,t\right)\E_{x_0}\left[f_{1}\left(Y_{s_{1}}^{(t)}\right)\ldots  f_{n}\left(Y_{s_{k}}^{(t)}\right)\right].
\end{align*}

First, $\mathcal{H}_1$ holds by definition \eqref{eq:semigroup_auxi}. Assuming that $\mathcal{H}_{k-1}$ is true for some $k>1$, we now prove $\mathcal{H}_k$. Let $0<s_{1}\leq s_{2}\leq\ldots  \leq s_{k}\leq t$ and $f_1,\ldots,f_k $ be measurable non-negative functions such that $f_i\in \mathcal{D}(\mathcal{A})$ for all $1\leq i\leq k$. Using the Markov property, we have
\begin{align*}
 & \E_{\delta_{x_0}}\left[\sum_{u\in V_{t}}f_{1}\left(X_{s_{1}}^{u}\right)\ldots  f_{k}\left(X_{s_{k}}^{u}\right)\right]\nonumber \\
 & =\E_{\delta_{x_0}}\left[\sum_{u\in V_{s_{k-1}}}f_{1}\left(X_{s_{1}}^{u}\right)\ldots  f_{k-1}\left(X_{s_{k-1}}^{u}\right)\E\left[\sum_{v\in V_{t},v\geq u}
f_{k}\left(X_{s_{k}}^{v}\right)\middle|\mathcal{F}_{s_{k-1}}\right]\right]\nonumber \\
 & =\E_{\delta_{x_0}}\left[\sum_{u\in V_{s_{k-1}}}f_{1}\left(X_{s_{1}}^{u}\right)\ldots  f_{k-1}\left(X_{s_{k-1}}^{u}\right)\E\left[\sum_{v\in V_{t}}f_{k}\left(X_{s_{k}}^{v}\right)\middle|Z_{s_{k-1}}=\delta_{X_{s_{k-1}}^{u}}\right]\right].\nonumber 
 \end{align*}
 We can now use the result proved in the case $k=1$ and the last term on the right hand side is equal to
 \begin{align*}
  & \E_{\delta_{x_0}}\left[\sum_{u\in V_{s_{k-1}}}\prod_{i=1}^{k-1}f_{i}\left(X_{s_{i}}^{u}\right)m\left(X_{s_{k-1}}^{u},s_{k-1},t\right)\E\left[f_{k}\left(Y_{s_{k}}^{\left(t\right)}\right)\Big|Y_{s_{k-1}}^{(t)}=X_{s_{k-1}}^{u}\right]\right]\nonumber \\
 &= \E_{\delta_{x_0}}\left[\sum_{u\in V_{t}}f_{1}\left(X_{s_{1}}^{u}\right)\ldots  f_{k-1}\left(X_{s_{k-1}}^{u}\right)\E\left[f_{k}\left(Y_{s_{k}}^{\left(t\right)}\right)\middle|Y_{s_{k-1}}^{(t)}=X_{s_{k-1}}^{u}\right]\right]\nonumber \\
 & =m(x_0,0,t)\E_{x_0}\left[f_{1}\left(Y_{s_{1}}^{(t)}\right)\ldots  f_{k-1}\left(Y_{s_{k-1}}^{(t)}\right)\E\left[f_{k}\left(Y_{s_{k}}^{\left(t\right)}\right)\middle|Y_{s_{k-1}}^{(t)}\right]\right],
\end{align*}
where the last equality is obtained using the induction hypothesis.

Finally, using Assumption \ref{assu:differentiable} and a monotone-class argument, we extend the result to all measurable function with respect to the Skorokod topology (see details in Appendix \ref{app:monotone}).
\end{proof}

We now develop two other Many-to-One formulas: one to characterize the trait of the individuals over the whole tree and the other to characterize the trait of a couple of individuals.
\subsection{A Many-to-One formula for the whole tree\label{sub:mtotree}}

We denote by:
\[\mathcal{T}=\bigcup_{s\geq 0} V_s\subset \mathcal{U},\]
the set of all individuals in the population. For $u\in\mathcal{\mathcal{T}}$, we denote by $\alpha(u)$ and $\beta(u)$ the random variables
representing respectively the time of birth and death of $u$. 
\begin{prop}\label{prop:mtotree}
Under Assumptions \ref{assu:debut1},\ref{assu:debut2}, \ref{assu:doeblin} and \ref{assu:differentiable}, for all $x_0\in\mathcal{X}$ and for any non-negative measurable function $F:\mathbb{D}\left(\mathbb{R}_+,\mathcal{X}\right)\times\mathbb{R}_{+}\rightarrow\mathbb{R}_+$, 
we have
\begin{equation}\label{eq:mtoarbre}
\E_{\delta_{x_0}}\left[\sum_{u\in\mathcal{T}}F\left(X_{[0,\beta(u))}^{u},\beta(u)\right)\right]=\int_{0}^{+\infty}m(x_0,0,s)\E_{x_0}\left[F\left(Y_{[0,s)}^{(s)},s\right)B\left(Y_{s}^{(s)}\right)\right]ds.
\end{equation}
\end{prop}
The left-hand side of \eqref{eq:mtoarbre} describes the dynamic of the trait of all individuals that were in the population. The right-hand side is the equivalent in terms of auxiliary process. Then, according to this result, the sum of the contributions of all individuals in the population is equal to the average of the auxiliary process with respect to the mean number of individuals in the population. The weight $B$ in the right-hand side comes from the density of the lifetimes. The terms might be infinite.
\begin{proof}
We follow \cite[Lemma 3.8]{cloez} and provide a proof for the whole trajectories. First, we recall that for any $u\in \mathcal{T}$ and any Borel set $A\subset \mathbb{R}_+$
\[
\mathbb{P}\left(\beta(u)\in A\left|\left(X_{s}^{u}\right)_{s\geq 0},\alpha(u) \right.\right)=\int_A B\left(X_{t}^{u}\right)\exp\left(-\int_{\alpha(u)}^{t}B\left(X_{s}^{u}\right)ds\right)dt.
\]
Then, for all non-negative measurable functions $f:\mathbb{D}(\mathbb{R}_+,\mathcal{X})\rightarrow \mathbb{R}_+$, we have
\begin{align*}
 & \E_{\delta_{x_0}}\left[\mathbf{1}_{\left\{ u\in\mathcal{T}\right\} }\int_{\alpha(u)}^{\beta(u)}F\left(X_{[0,s)}^{u},s\right)B\left(X_{s}^{u}\right)ds\right]\\
= & \E_{\delta_{x_0}}\Bigg[\mathbf{1}_{\left\{ u\in\mathcal{T}\right\} }\int_{\alpha(u)}^{+\infty}\left(\int_{\alpha(u)}^{\tau}F\left(X_{[0,s)}^{u},s\right)B\left(X_{s}^{u}\right)ds\right)\\
& \hspace{7cm}\times B\left(X_{\tau}^{u}\right)\exp\left(-\int_{\alpha(u)}^{\tau}B\left(X_{r}^{u}\right)dr\right)d\tau\Bigg].
\end{align*}
Next, using Fubini's Theorem, we obtain that the right-hand side of the above equation is equal to
\begin{align}\nonumber
&\E_{\delta_{x_0}}\Bigg[\mathbf{1}_{\left\{ u\in\mathcal{T}\right\} }\int_{\alpha(u)}^{+\infty}\left(\int_{s}^{+\infty}B\left(X_{\tau}^{u}\right)\exp\left(-\int_{\alpha(u)}^{\tau}B\left(X_{r}^{u}\right)dr\right)d\tau\right)\\\nonumber
&\hspace{7cm}\times F\left(X_{[0,s)}^{u},s\right)B\left(X_{s}^{u}\right)ds\Bigg]\\\nonumber
= & \E_{\delta_{x_0}}\left[\mathbf{1}_{\left\{ u\in\mathcal{T}\right\} }\int_{\alpha(u)}^{+\infty}\exp\left(-\int_{\alpha(u)}^{s}B\left(X_{r}^{u}\right)dr\right)F\left(X_{[0,s)}^{u},s\right)B\left(X_{s}^{u}\right)ds\right]\\\label{thewohle1}
= & \E_{\delta_{x_0}}\left[\mathbf{1}_{\left\{ u\in\mathcal{T}\right\} }F\left(X_{[0,\beta(u))}^{u},\beta(u)\right)\right],
\end{align}
where the first equality is comes from of Assumption \ref{assu:debut1}(4).
But $$\left\{ \alpha(u)\leq s<\beta(u),u\in\mathcal{T}\right\} =\left\{ u\in V_{s}\right\}, $$
then,
\begin{align}\nonumber
\E_{\delta_{x_0}}&\left[\mathbf{1}_{\left\{ u\in\mathcal{T}\right\} }\int_{\alpha(u)}^{\beta(u)}F\left(X_{[0,s)}^{u},s\right)B\left(X_{s}^{u}\right)ds\right] \\\label{eq:thewohle2}
& =\E_{\delta_{x_0}}\left[\int_{0}^{+\infty}\mathbf{1}_{\left\{ u\in V_{s}\right\} }F\left(X_{[0,s)}^{u},s\right)B\left(X_{s}^{u}\right)ds\right].
\end{align}
Finally combining \eqref{thewohle1} and \eqref{eq:thewohle2} we get
\begin{align*}
\E_{\delta_{x_0}}\left[\sum_{u\in\mathcal{T}}f\left(X_{[0,\beta(u))}^{u},\beta(u)\right)\right] & =\sum_{u\in\mathcal{U}}\E_{\delta_{x_0}}\left[\mathbf{1}_{\left\{ u\in\mathcal{T}\right\} }F\left(X_{[0,\beta(u))}^{u},\beta(u)\right)\right]\\
 & =\sum_{u\in\mathcal{U}}\E_{\delta_{x_0}}\left[\int_{0}^{+\infty}\mathbf{1}_{\left\{ u\in V_{s}\right\} }F\left(X_{[0,s)}^{u},s\right)B\left(X_{s}^{u}\right)ds\right]\\
 & =\int_{0}^{+\infty}\E_{\delta_{x_0}}\left[\sum_{u\in V_{s}}F\left(X_{[0,s)}^{u},s\right)B\left(X_{s}^{u}\right)\right]ds\\
 & =\int_{0}^{+\infty}m(x_0,0,s)\E_{x_0}\left[F\left(Y_{[0,s)}^{(s)},s\right)B\left(Y_{s}^{(s)}\right)\right]ds,
\end{align*}
where the last equality comes from the Many-to-One formula \eqref{mtomesurable}. 
\end{proof}

\subsection{Many-to-One formulas for forks\label{sub:mtoforks}}
In this section, we characterize the law of a couple of lineage coming from two individuals alive at time $t$. For former results on the subject, we refer to \cite{bansaye2011limit} for such formulas in the neutral case and to \cite{harris2015many} for many-to-few-formulas on weighted $k$-fold sums over particles in the case of local branching. We aim at characterizing the dynamic of the trait of a couple of individual along the spine using our auxiliary process. Those formulas have already proved useful to control the variance of estimators \cite{hoffmann2015nonparametric}.

For any two functions $f,g$, defined respectively on two intervals $I_f$, $I_g$, for any $[a,b)\subset I_f$,  $[c,d)\subset I_g$, we define the concatenation $[f_{[a,b)},g_{[c,d)}]$ by
\[
[f_{[a,b)},g_{[c,d)}](t)=\left\{\begin{array}{ll}
f(t), & \text{ if }t\in [a,b),\\
g(t+c-b), & \text{ if }t\in [b,b+(d-c)).\\
\end{array}\right.
\]
\begin{prop}\label{prop:mtoforksgene}
Under Assumptions \ref{assu:debut1}, \ref{assu:debut2}, \ref{assu:doeblin} and \ref{assu:differentiable}, for any $t\geq 0$ and $x_0\in \mathcal{X}$, for any non-negative measurable function $F:\mathbb{D}([0,t],\mathcal{X})^2\rightarrow\mathbb{R}_+$,
\begin{align}\label{eq:mtoforksgene}
\E_{\delta_{x_0}}\left[\sum_{\begin{subarray}{c}
u,v\in V_{t}\\
u\neq v
\end{subarray}}F\left(X_{[0,t]}^{u},X_{[0,t]}^{v}\right)\right]=\int_0^t m(x_0,0,s)\mathbb{E}_{x_0}\left[B\left(Y_s^{(s)}\right)J_{s,t}F\left(Y_{[0,s]}^{(s)}\right)\right]ds,
\end{align}
where for $(x_r,r\leq s)\in\mathbb{D}([0,s],\mathcal{X})$,
\begin{multline*}
J_{s,t}F(x)=\sum_{a\neq b\in\mathbb{N}}\sum_{k\geq \max(a,b)}p_k\left(x_s\right)\int_0^1 m\left(F_a^{(k)}\left(x_s,\theta\right),s,t\right)m\left(F_b^{(k)}\left(x_s,\theta\right),s,t\right)\\
H_{s,t}F\left(x,F_a^{(k)}\left(x_s,\theta\right),F_b^{(k)}\left(x_s,\theta\right)\right)d\theta,
\end{multline*}
and for all $s\leq t$ , $\left(x_s,s\leq t\right)\in\mathbb{D}\left([0,t],\mathcal{X}\right)$ and $y_1,y_2\in\mathcal{X}$,
\begin{align*}
H_{s,t}F(x,y_1,y_2)=\mathbb{E}\left[F\left([x_{[0,s)};Y_{[s,t]}^{(t),1}],[x_{[0,s)};Y_{[s,t]}^{(t),2}]\right)\middle|\left(Y_s^{(t),1},Y_s^{(t),2}\right)=\left(y_1,y_2\right)\right],
\end{align*}
and $(Y^{(t),1}_s,\ s\leq t),(Y^{(t),2}_s,\ s\leq t)$ are two independent copies of $(Y^{(t)}_s,\ s\leq t)$. 
\end{prop}
\begin{figure}
\centering
\def\svgwidth{0.45\textwidth}
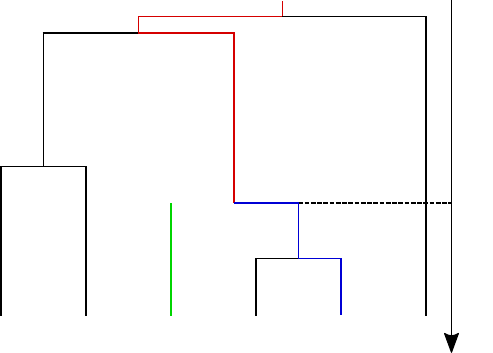
\caption{Forks.}
\label{fig:forks1}
\end{figure}
According to this proposition, the sum of the contributions of each couple in the population at time $t$ corresponds to an integral of a product of a count term and a term characterizing the dynamic of the traits of the couple. The integral is over all the possible death time $s\in [0,t]$ for the most recent common ancestor $w$ of $u$ and $v$, where $u\neq v\in V_t$. For example, in the case of Figure \ref{fig:forks1}, if we pick the green star and the blue star, the lineage of their most recent common ancestor $w$ is in red. For the count term, $m(x_0,0,s)$ corresponds to the choice $w$ among the individuals in the population at time $s$ and $m\left(F_a^{(k)}\left(x_s,\theta\right),s,t\right)m\left(F_b^{(k)}\left(x_s,\theta\right),s,t\right)$ corresponds to the choice of $u$ and $v$ among the descendants of $w$. In the example, with our choice of $w$, there is only one choice for $u$ and two for $v$. Before $s$, the traits along the ancestral lineage of $u$ and $v$ are identical. After the death of $w$, the dynamic of the trait of the ancestor of $u$ and the ancestor of $v$ become independent conditionally to the trait of $w$ at death. This explains the term $H_{s,t}F\left(x_{[0,s)},F_a^{(k)}\left(x_s,\theta\right),F_b^{(k)}\left(x_s,\theta\right)\right)$ above and, in the right
hand side of \eqref{eq:mtoforksgene}, $s$ represents the time of the most recent common ancestor
with $x_{[0,s)},\ Y^{(t),1}_{[s,t]}$ and $Y^{(t),2}_{[s,t]}$ describing the dynamics of the trait along the red, green and blue path respectively.

This formula is similar to the Many-to-Two formula proved in \cite{harris2015many} but as in the Many-to-One formula \eqref{mtomesurable}, the count terms are separated from the terms corresponding to the dynamic of the trait of a "typical" individual contrary to the formula in \cite{harris2015many}. This decomposition is useful for the study of the asymptotic behavior of the branching process. We refer the reader the \cite{marguet2017law} for an example of use of this formula to prove a law of large numbers.
\begin{proof}
Let $t\geq 0$ and $x_0\in\mathcal{X}$. First we prove \eqref{eq:mtoforksgene} for $F(x,y)=f_1(x)f_2(y)$, where $f_i:\mathbb{D}([0,t],\mathcal{X})\rightarrow \mathbb{R}_+$ are non-negative measurable functions for $i=1,2$. Let us denote by $A$ the left-hand side of \eqref{eq:mtoforksgene}.  We explicit the most recent common ancestor $w$ of two individuals $u,v$ living at time $t$ and we obtain
\begin{multline*}
A=\E_{\delta_{x_0}}\Bigg[\sum_{w\in\mathcal{U}}\sum_{a_1\neq a_2\in\mathbb{N}}\sum_{\widetilde{u}_1,\widetilde{u}_2\in\mathcal{T}}\mathbf{1}_{\left\{ t\geq\beta(w)\right\} }\mathbf{1}_{\left\{wa_1\widetilde{u}_1\in V_t,\ wa_2\widetilde{u}_2\in V_t\right\}}\\
\times\prod_{i=1,2}f_i\left(\left[X_{[0,\beta(w))}^{w};X_{[\beta(w),t]}^{wa_i\widetilde{u}_i}\right]\right)\Bigg]\\
=\E_{\delta_{x_0}}\left[\sum_{\begin{subarray}{c} wa_1\neq wa_2\in\mathcal{\mathcal{T}}\\ a_1,a_2\in\mathbb{N}\end{subarray}}\mathbf{1}_{\left\{ t\geq\beta(w)\right\} }\E\left[\prod_{i=1,2}\sum_{u_i\in V_{t},u_i\geq wa_i}f_i\left(\left[X_{[0,\beta(w))}^{w};X_{[\beta(w),t]}^{u_i}\right]\right)\middle| \mathcal{F}_{\beta(w)}\right]\right].
\end{multline*}
Then, applying successively the branching property and the Markov property, we have
\begin{multline*}
A = \E_{\delta_{x_0}}\Bigg[\sum_{\begin{subarray}{c} wa_1\neq wa_2\in\mathcal{\mathcal{T}}\\ a_1,a_2\in\mathbb{N}\end{subarray}}\mathbf{1}_{\left\{ t\geq\beta(w)\right\} }\prod_{i=1,2}\E\Bigg[\sum_{\begin{subarray}{c} u_i\in V_{t}\\ u_i\geq wa_i\end{subarray}}f_i\left(\left[X_{[0,\beta(w))}^{w};X_{[\beta(w),t]}^{u_i}\right]\right)\Bigg|X_{[0,\beta(w)]}^{wa_i}\Bigg]\Bigg]\\
= \E_{\delta_{x_0}}\Bigg[\sum_{\begin{subarray}{c} wa_1\neq wa_2\in\mathcal{\mathcal{T}}\\ a_1,a_2\in\mathbb{N}\end{subarray}}\mathbf{1}_{\left\{ t\geq\beta(w)\right\} }\prod_{i=1,2}\E\Bigg[\sum_{\begin{subarray}{c} u_i\in V_{t}\\ u_i\geq wa_i\end{subarray}}f_i\left(\left[\widetilde{x};X_{[\beta(w),t]}^{u_i}\right]\right)\Bigg|X_{\beta(w)}^{wa_i}\Bigg]_{\widetilde{x}=X_{[0,\beta(w))}^{w}}\Bigg].
\end{multline*}
Next, we use the Many-to-One formula \eqref{mtomesurable} and we explicit the distribution of the trait at birth of $wa$ and $wb$:

\begin{align*}
&\E_{\delta_{x_0}}\Bigg[\sum_{\begin{subarray}{c} wa_1\neq wa_2\in\mathcal{\mathcal{T}}\\ a_1,a_2\in\mathbb{N}\end{subarray}}\mathbf{1}_{\left\{ t\geq\beta(w)\right\} }\prod_{i=1,2}m\left(X_{\beta(w)}^{wa_i},\beta(w),t\right) \\
&\hspace{4cm}\times\E\left[f_i\left(\left[\widetilde{x};Y_{[\beta(w),t]}^{(t)}\right]\right)\Big|Y_{\beta(w)}^{(t)}=X_{\beta(w)}^{wa_i}\right]_{\widetilde{x}=X_{[0,\beta(w))}^{w}}\Bigg]\\
& = \E_{\delta_{x_0}}\Bigg[\sum_{\begin{subarray}{c} wa_1\neq wa_2\in\mathcal{\mathcal{T}}\\ a_1,a_2\in\mathbb{N}\end{subarray}}\mathbf{1}_{\left\{ t\geq\beta(w)\right\} }\sum_{k\geq a_1\vee a_2}p_{k}\left(X_{\beta(w)}^{w}\right)\\
&\hspace{2cm}\times \int_{0}^{1}\prod_{i=1,2}m\left(F_{a_i}^{(k)}\left(X_{\beta(w)}^{w},\theta\right),\beta(w),t\right) \\
&\hspace{3cm}\times\E\left[f_i\left(\left[\widetilde{x};Y_{[\beta(w),t]}^{(t)}\right]\right)\Big|Y_{\beta(w)}^{(t)}=F_{a_i}^{(k)}\left(X_{\beta(w)^-}^{w},\theta\right)\right]_{\widetilde{x}=X_{[0,\beta(w))}^{w}}d\theta\Bigg].
 \end{align*}
Applying the Many-to-One formula over the whole tree \eqref{eq:mtoarbre} yields
\begin{align*}
A & =  \int_{0}^{t}m(x_0,0,s)\E_{x}\Bigg[B\left(Y_{s}^{(s)}\right)\sum_{a_1\neq a_2\in\mathbb{N}}\sum_{k\geq a_1\vee a_2}p_{k}\left(Y_{s}^{(s)}\right)\\
&\hspace{3cm}\times \int_{0}^{1}\prod_{i=1,2}m\left(F_{a_i}^{(k)}\left(Y_{s}^{(s)},\theta\right),s,t\right)\\
&\hspace{4cm}\times\E\left[f_i\left(\left[\widetilde{x};Y_{[s,t]}^{(t)}\right]\right)\Big|Y_{s}^{(t)}=F_{a_i}^{(k)}\left(Y_s^{(s)},\theta\right)\right]_{|\widetilde{x}=Y_{[0,s)}^{(s)}}d\theta\Bigg]ds\\
& =  \int_{0}^{t}m(x_0,0,s)\E_{x_0}\Bigg[B\left(Y_{s}^{(s)}\right) J_{s,t}\left(f_1\otimes f_2\right)\left(Y_{[0,s)}^{(s)}\right)\Bigg]ds,
\end{align*}
where $f_1\otimes f_2(x)=f_1(x)f_2(x)$.
Finally, we obtain \eqref{eq:mtoforksgene} using a monotone class argument. 
\end{proof}
Let us explicit a particular case of formula \eqref{eq:mtoforksgene}. We define
\begin{equation}
J_{2}(f,g)(x)=\sum_{a\neq b}\sum_{k\geq\max(a,b)}p_{k}\left(x\right)\int_{0}^{1}f\left(F_{a}^{(k)}\left(x,\theta\right)\right)g\left(F_{b}^{(k)}\left(x,\theta\right)\right)d\theta.\label{eq:J2}
\end{equation}
$J_{2}$ represents the average trait at birth of two uniformly chosen children
from an individual of type $x$. For simplicity of notation, we write $J_2f(x)$ instead of $J_2(f,f)(x)$. Let us recall that
\begin{align*}
P_{r,s}^{(t)}f(x)=\mathbb{E}\left[f\left(Y_s^{(t)}\right)\middle|Y_r^{(t)}=x\right].
\end{align*}
\begin{cor}
\label{mtoforksst}Under Assumptions \ref{assu:debut1},\ref{assu:debut2}, \ref{assu:doeblin} and \ref{assu:differentiable}, for any non-negative measurable functions $f_t,g_t$ from $\mathcal{X}\times\mathbb{R}^{+}$ to $\mathbb{R}$ and any $x_0\in\mathcal{X}$
we have for $s\leq t$,
\begin{align}\nonumber
&\E_{\delta_{x_0}}\left[\sum_{\begin{subarray}{c}
u,v\in V_{t}\\
u\neq v
\end{subarray}}f_{t}\left(X_{s}^{u}\right)g_{t}\left(X_{s}^{v}\right)\right]\\\nonumber
&= \int_{s}^{t}m(x_0,0,r) \E_{x_0}\left[f_{t}\otimes g_{t}\left(Y_{s}^{(r)}\right)B\otimes J_{2}m\left(\cdot,r,t\right)\left(Y_{r}^{(r)}\right)\right]dr\\\label{eq:mtoforksplus}
&\hspace{1cm}+\int_{0}^{s}m(x_0,0,r)\E_{x_0}\left[B\otimes J_{2}\left(m(\cdot,r,t)P_{r,s}^{(t)}f_{t},m(\cdot,r,t)P_{r,s}^{(t)}g_{t}\right)\left(Y_{r}^{(r)}\right)\right]dr.
\end{align}
\end{cor}
The first integral corresponds to the couple of individuals alive at time $t$ whose most recent common ancestor died after time $s$. It is for example the case on Figure \ref{fig:forks2} if you pick two red stars on the tree on the left-hand side. The product $m(x_0,0,r)J_2m(\cdot,r,t)(y)$, with $y\in\mathcal{X}$, corresponds to the average number of such couples at time $t$ whose most recent common ancestor died at time $r$ with $s\leq r\leq t$. The second integral corresponds to couples $(u,v)\in V_t$ of individuals whose most recent common ancestor $w$ died before $s$. It is the case on Figure \ref{fig:forks2} if you pick one blue star and one green star on the tree on the right-hand side. In this case, unlike in the previous one, the value of the trait of the individuals at time $s$ is not the same. The dynamic of the trait on the blue lineage and on the green lineage are independent conditionally to the trait of their common ancestor at death. This explains the terms $P_{r,s}^{(t)}f_t$ and $P_{r,s}^{(t)}g_t$ that appear in the second integral. As before, the remaining terms depending on the average number of individuals in the population are count terms.
\begin{figure}
\centering
\def\svgwidth{0.45\textwidth}
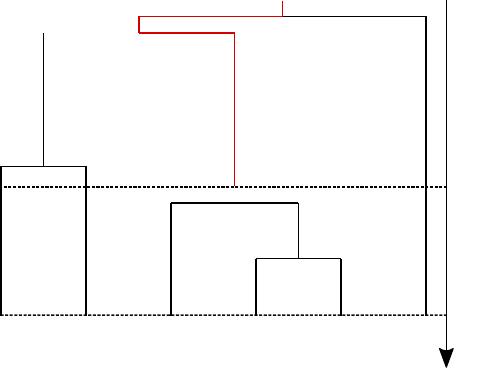
\def\svgwidth{0.45\textwidth}
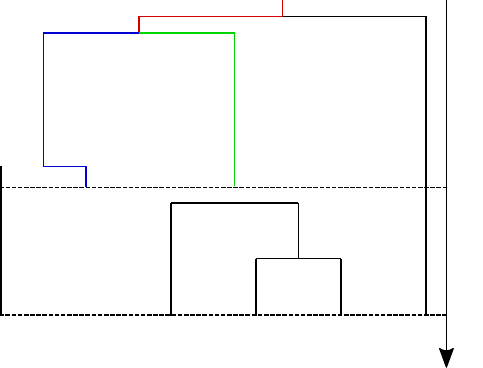
\caption{Trees and forks.}
\label{fig:forks2}
\end{figure}

\section{Ancestral lineage of a uniform sampling at a fixed time with a large initial population}\label{infiniteparticle}\label{sec:large}
The Many-to-One formula \eqref{mtomesurable} gives us the law of the trait of a uniformly sampled individual in an "average" population. But the characterization of the law of the trait of a uniformly sampled individual in the effective population is more complex because the number of individuals alive at time $t$ is stochastic and depends on the dynamic of the trait of individuals. As the auxiliary process takes into account the bias in the population due to the number of individuals, it characterizes the law of a uniformly sampled individual only when the bias are in place i.e. when there are a certain amount of individuals. Indeed, the dynamic of the first individual in the population is not biased. That is why we now look at the ancestral lineage of a uniform sampling with a large initial population.
\subsection{Convergence of the sampling process on a fixed time interval}

It only makes sense to speak of a uniformly sampled individual at time $t$ if the population does not become extinct before time $t$. For all $t\geq 0$, let $\Omega_t=\left\lbrace N_t>0\right\rbrace$ denote the event of survival of the population. Let $\nu\in\mathcal{M}_P(\mathcal{X})$ be such that
\begin{align}\label{nonextinct}
\mathbb{P}_{\nu}(\Omega_t)>0.
\end{align}
We set
\begin{align}\label{eq:nun}
\nu_n:=\sum_{i=1}^n\delta_{X_i},
\end{align}
where $X_i$ are i.i.d. random variables with distribution $\nu$.
For $t\geq 0$, we denote by $U(t)$ the random variable with uniform distribution on $V_t$ conditionally on $\Omega_t$ and by $\left(X^{U(t)}_s,\ s\leq t\right)$ the process describing the trait of a sampling along its ancestral lineage. 
If $X$ is a stochastic process, we denote by $X^{\nu}$ the process with initial distribution $\nu\in\mathcal{M}_P(\mathcal{X})$. In particular, for all $t\geq 0$, $Y^{(t),\nu}$ corresponds to the auxiliary process with $Y_0^{(t)}\sim \nu$. For all $0\leq s\leq t$,
\[m(\nu,s,t)=\E\left(N_{t}| Z_{s}=\nu\right),\]
denote the average number of individuals in the population after time $t$ starting from a population distributed as $\nu$ at time $s$. 
\begin{thm}\label{th:grdpop} Under Assumptions \ref{assu:debut1}(1-3),\ref{assu:debut2}, \ref{assu:doeblin} and \ref{assu:differentiable}, for any $t\geq 0$, the sequence $\left(X_{[0,t]}^{U(t),\nu_n}\right)_{n\geq 0}$ converges in law in $\mathbb{D}\left([0,t],\mathcal{X}\right)$ to $Y_{[0,t]}^{(t),\pi_t}$ where
\begin{align*}
\pi_t(dx)=\frac{m(x,0,t)\nu(dx)}{m(\nu,0,t)}.
\end{align*}
\end{thm}

\begin{proof}

Let $t\geq 0$. Let $(X_i)_{1\leq i\leq n}$ be i.i.d random variables with distribution $\nu$ and $\nu_n=\sum_{i=1}^{n}\delta_{X_i}$. Let $F:\mathbb{D}\left([0,t],\mathcal{X}\right)\rightarrow\mathbb{R}_+$ be a bounded measurable function. First, we notice that
\begin{align}\label{eq:lawlarge}
\frac{1}{n}N_t^{\nu_n}=\frac{1}{n}\sum_{i=1}^nN_t^{(i)},
\end{align}
where $N_t^{(i)}$ are independent copies of $N_t$ with initial distribution $\delta_{X_i}$. According to the law of large numbers, \eqref{eq:lawlarge} converges almost surely as $n$ tends to infinity to $m(\nu,0,t)=\int_{\mathcal{X}}m(x,0,t)\nu(dx)$. Next, let $\Omega_t(\nu_n)=\left\{N_t^{\nu_n}>0\right\}$. $\left(\Omega_t(\nu_n)\right)_{n\geq 0}$ is a increasing sequence. According to \eqref{nonextinct}, there exists $0<\varepsilon(t)\leq 1$ such that
\[\mathbb{P}(\Omega_t(\nu_n)^C)\leq(1-\varepsilon(t))^n\underset{n\rightarrow +\infty}{\longrightarrow} 0,\]
so that:
\begin{align*}
\mathbf{1}_{\left\{\Omega_t(\nu_n)^C\right\}}\underset{n\rightarrow +\infty}{\longrightarrow} 0, \text{ almost surely}. 
\end{align*}  We have
\begin{align*}
\mathbb{E}\left[F\left(X_{[0,t]}^{U(t),\nu_n}\right)\right]=\mathbb{E}\left[\mathbf{1}_{\left\{\Omega_t(\nu_n)\right\}}\frac{1}{N_t^{\nu_n}}\sum_{u\in V_t^{\nu_n}} F\left(X_{[0,t]}^u\right)\right]\mathbb{P}\left(\Omega_t(\nu_n)\right)^{-1}.
\end{align*}
Let $V_t^{(i)},i=1\ldots n$ be independent identically distributed populations at time $t$ coming from an individual with trait $X_i\sim \nu$ at $0$. Then,
\begin{align}\label{eq:limit}
\mathbb{E}\left[\mathbf{1}_{\left\{\Omega_t(\nu_n)\right\}}\frac{1}{N_t^{\nu_n}}\sum_{u\in V_t^{\nu_n}}F\left(X_{[0,t]}^u\right)\right]&=\mathbb{E}\left[\mathbf{1}_{\left\{\Omega_t(\nu_n)\right\}}\frac{n}{N_t^{\nu_n}}\frac{1}{n}\sum_{i=1}^n\sum_{u\in V_t^{(i)}}F\left(X_{[0,t]}^u\right)\right].
\end{align}
According to the strong law of large numbers, 
\begin{align*}
\frac{1}{n}\sum_{i=1}^n\sum_{u\in V_t^{(i)}}F\left(X_{[0,t]}^u\right)\underset{n\rightarrow +\infty} {\longrightarrow} \E_{\nu}\left[\sum_{u\in V_t}F\left(X_{[0,t]}^u\right)\right],\text{ almost surely.}
\end{align*}
Taking the limit in \eqref{eq:limit} as $n$ tends to infinity, by dominated convergence because $F$ is bounded, we have
\begin{align*}
\mathbb{E}\left[F\left(X_{[0,t]}^{U(t),\nu_n}\right)\right]\underset{n\rightarrow +\infty} {\longrightarrow}\frac{1}{m(\nu,0,t)}\int_{\mathcal{X}}\E_{x}\left[\sum_{u\in V_t}F\left(X_{[0,t]}^u\right)\right]\nu(dx),
\end{align*}
because $\mathbb{P}\left(\Omega_t(\nu_n)\right)\rightarrow 1$ as $n$ tends to infinity. Finally, applying the Many-to-one formula \eqref{mtomesurable}, we obtain
\begin{align*}
\mathbb{E}\left[F\left(X_{[0,t]}^{U(t),\nu_n}\right)\right]\underset{n\rightarrow +\infty} {\longrightarrow}\frac{\int_{\mathcal{X}}m(x,0,t)\mathbb{E}_x\left[F\left(Y_{[0,t]}^{(t)}\right)\right]\nu(dx)}{m(\nu,0,t)}.
\end{align*}
\end{proof}
\begin{rem}
If we start with $n$ individuals with the same trait $x$, we obtain
\begin{align*}
\mathbb{E}\left[F\left(X_{[0,t]}^{U(t),\nu_n}\right)\right]\underset{n\rightarrow +\infty} {\longrightarrow}\mathbb{E}_x\left[F\left(Y_{[0,t]}^{(t)}\right)\right].
\end{align*}
Therefore, the auxiliary process describes exactly the dynamic of the trait of a uniformly sampled individual in the large initial population limit, if all starting individuals have the same trait. 
If the initial individuals have different traits at the beginning, the large population approximation of a uniformly sampled individual is a linear combination of the auxiliary process. 
\end{rem}
\begin{rem}
One can easily generalizes this results to a $k$-tuple of individuals uniformly picked at time $t$. If you start with a population of size $n$ and you pick $k$ individuals uniformly at random at time $t$, when $n$ tends to infinity, the probability that those $k$ individuals comes from the same initial individual is zero. Then, the trajectories of their traits are independent and for example in the case $k=2$, for any $f,g:\mathbb{D}\left([0,t],\mathcal{X}\right)\rightarrow\mathbb{R}_+$ bounded measurable functions, we get
\begin{align*}
\mathbb{E}\left[f\left(X_{[0,t]}^{U_1(t),\nu_n}\right)g\left(X_{[0,t]}^{U_2(t),\nu_n}\right)\right]\underset{n\rightarrow +\infty} {\longrightarrow}\mathbb{E}_x\left[f\left(Y_{[0,t]}^{(t),1}\right)g\left(Y_{[0,t]}^{(t),2}\right)\right],
\end{align*}
where $U_1(t),U_2(t)$ are independent random variables with uniform distribution on $V_t$ and the processes $\left(Y_s^{(t),1},s\leq t\right), \left(Y_s^{(t),2},s\leq t\right)$ are i.i.d. and distributed as $\left(Y_{s}^{(t)},s\leq t\right).$ 
\end{rem}
\begin{rem}
An other way of characterizing the trait of a uniformly sampled individual via the auxiliary process is to look at the long time behavior of the process instead of looking at the large initial population behavior. This has been done in \cite{bansaye2011limit} in the case of a constant division rate and in \cite{marguet2017law} in a general framework using the ergodicity of ancestral lineages.  
\end{rem}
\subsection{The trait of a uniformly sampled individual for growth-fragmentation models}\label{sec:exfin}
The auxiliary process is a good way of getting simulated random variables corresponding to the trait of a uniformly sampled individual with large initial population or to the trait of a uniformly sampled individual for large times (see \cite{marguet2017law}). Indeed, it is much more quicker to simulate one trajectory of the auxiliary process rather than the dynamic of an entire population. In this section, we detail the auxiliary process for our three examples introduced in Section \ref{sec:exdebut}.
\subsubsection{Linear growth model}\label{sub:TCP3}
For the linear growth model with binary division (Section \ref{sub:TCP}), Assumption \ref{assu:doeblin} is satisfied for $C\equiv 1$ and the large initial population limit of the ancestral process of a sampling grows linearly between two jumps and jumps at time $s$ at rate 
\[
\widehat{B}_{s}^{(t)}(x)=\alpha x\left(1+\frac{1+e^{2\overline{a}\left(t-s\right)}}{1-x\sqrt{\frac{\alpha}{a}}+e^{2\overline{a}\left(t-s\right)}\left(1+x\sqrt{\frac{\alpha}{a}}\right)}\right).
\]
At a jump, there is a unique descendant with trait $\frac{x}{2}$ if
$x$ is the trait of its parent at the splitting time. 
In particular, the rate of division of the limiting process is bigger than the rate of division in a cell line for the original process. It means that in the large initial population limit, a typical individual has overcome more division than any individual.

\subsubsection{Exponential growth model in a varying environment}\label{sub:cell3}
For the exponential growth model in a varying environment with binary division (Section \ref{sub:binary}), Assumption \ref{assu:doeblin} is satisfied for $C\equiv 1$ and  the associated auxiliary process grows exponentially between two jumps and jumps at time $s$ at rate 
\[
\widehat{B}_{s}^{(t)}(x)=\left(\alpha(s) x+\beta\right)\left(1+\frac{1}{1+x\int_s^t \alpha(r)e^{(a-\beta)(r-s)}dr}\right).
\]
The rate of division of the limiting process is again bigger than the division rate of any individual.
At a jump, there is a unique descendant with trait $\frac{x}{2}$ if
$x$ is the trait of its parent at the splitting time. 

This example is a good illustration of the fact that the large initial population limit of the size of a uniformly sampled individual does not correspond to the size of a tagged cell, i.e. the size along a lineage where at each division, you choose randomly one daughter cell. In fact, as the division rate of the auxiliary process is larger than $B$, the number of divisions along the lineage of a uniformly sampled individual is bigger than the number of divisions along the lineage of tagged cell, resulting in a difference on the size of the individuals. However, the distribution of the number of divisions along the lineage of a uniformly sampled individual coincides with the one for the auxiliary process. On Figure \ref{comparaison_trait}, we can see those distributions: the two first distributions, corresponding to the distribution of the number of divisions along the lineage of a uniformly sampled individual and of the auxiliary process, are centered on a bigger number of divisions than the third distribution corresponding to a tagged cell.  
\begin{figure} 
\centering
\includegraphics[width=8.5cm]{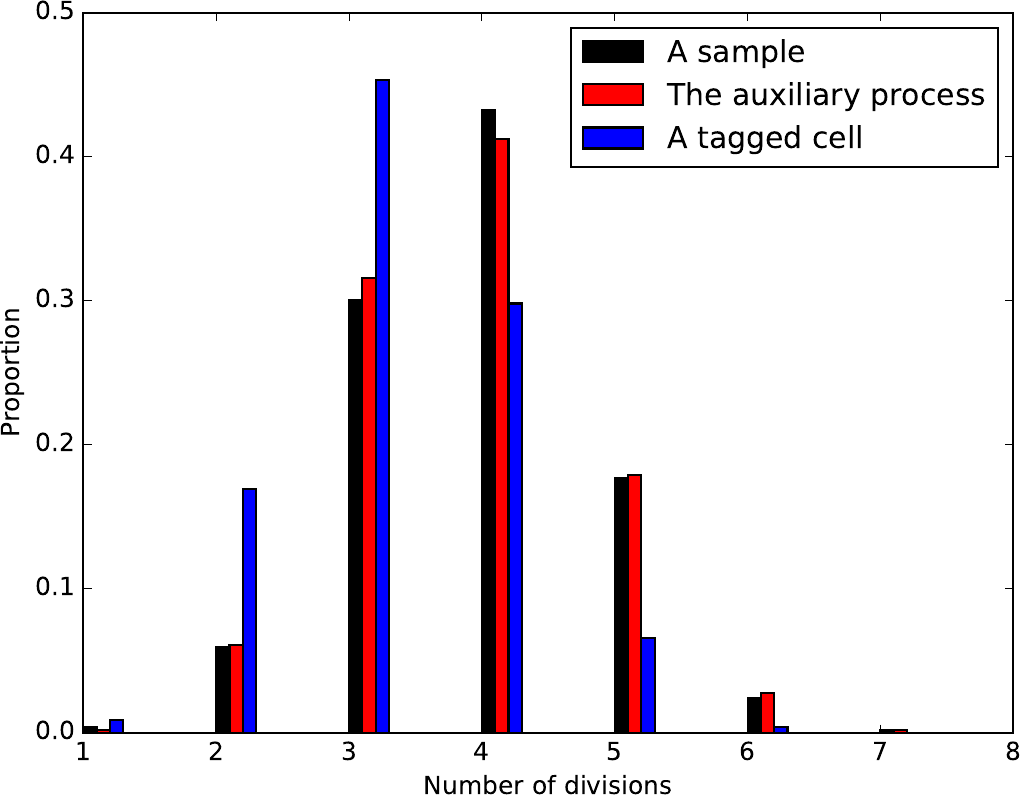} 
\caption{Distribution of the number of divisions in the lineage of a uniformly sampled individual (black bars), of the auxiliary process (red bars) and of a tagged cell (blue bars). For each case, we used $5000$ realizations of each process until time $t=50$ with parameters $a=0.1$ and $x_0=1$. The distribution of the number of divisions almost coincides for the auxiliary process and a sampled individual. However, the distribution of the number of divisions for a tagged cell is different from the two previous ones. Indeed, it is more likely to sample an individual whose ancestors divided many times, that is why, the distributions of the number of divisions for the auxiliary process and for a uniformly sampled individual are centered on bigger values than the distribution of the number of divisions for a tagged cell.}
\label{comparaison_trait}
\end{figure}
 
\subsubsection{Parasite infection model}\label{sub:kimmel2}
For this cell division model with parasite infection, Assumption \ref{assu:doeblin} is satisfied for $C\equiv 1$ and the auxiliary process evolves as a Feller diffusion with infinitesimal generator
\[
\mathcal{F}^{(t)}_sf(x)=\left(gx+2\sigma^2\frac{\alpha x\left(e^{g(t-s)}-e^{\beta(t-s)}\right)}{\alpha x\left(e^{g(t-s)}-e^{\beta(t-s)}\right)+(g-\beta)e^{\beta(t-s)}}\right)f'(x)+\sigma^2xf''(x),
\]
so that the drift of the limit of the process of the ancestral trait of a sampling is bigger than the original drift in the population. Then, the limiting process jumps at time $s$ at rate 
\[
\widehat{B}_{s}^{(t)}(x)=\left(\alpha x+\beta\right)\left(1+\frac{1}{1+\frac{\alpha x}{g-\beta}\left(e^{(g-\beta)\left(t-s\right)}-1\right)}\right).
\]
Therefore, the division rate of the limiting process is also bigger than the rate of division in a cell line for the original process. \\
The trait of the newborn cell is distributed according to the following probability law:
\[
\widehat{P}_{s}^{(t)}(x,dy)=\mathbf{1}_{\left\{0\leq y\leq x\right\}}\frac{2(g-\beta)+2\alpha y \left(e^{(g-\beta)\left(t-s\right)}-1\right)}{2(g-\beta)+\alpha x\left(e^{(g-\beta)\left(t-s\right)}-1\right)}\frac{dy}{x}.
\]
In fact, because cells divide faster when they have more parasites inside them, it is a good strategy, in order to have a lot of descendants in a long time scale, to choose to give a lot of parasites to your daughter cell. Moreover, the evolution of the trait is biased: the drift in the Feller diffusion is more important for the auxiliary process because a cell with more parasites divides faster so that it produces more descendants.
\section{Further comments and examples}\label{sec:other}
We can apply the results of this work to various models and we choose to detail in this article only three of them based on biological and computational considerations. However, we review in this section some other interesting models.

\subsection{The age-structured population model} In this model, the quantity of interest is the age of each individual which grows linearly. The lifetime of each individual is a random variable with cumulative distribution function $G$. Such models have been first introduced by Bellman and Harris in \cite{bellman1952age} and have recently been studied in order to infer the division rate \cite{hoffmann2015nonparametric}. Let $B:\mathbb{R}_+\rightarrow\mathbb{R}$ be the rate of division of each cell defined via
\[
G(t)=1-\exp\left(-\int_0^tB(s)ds\right).
\]
The branching process $(Z_t)_{t\geq 0}$ is solution of the following equation, for any function $f\in \mathcal{C}^1(\mathbb{R}_+)$ and any $x\in\mathcal{X}$: 
\begin{align*}
\langle Z_t,f\rangle=&\langle Z_0,f\rangle +\int_0^t \int_{\mathbb{R}_+} f'(x)Z_s(dx)ds\\
&+\int_0^t\int_{\mathcal{U}\times \mathbb{R}_+\times \mathbb{N}}\mathbf{1}_{\left\{u\in V_{s^-},\ \theta\leq B\left( X_{s^-}^u\right)\right\}}\left(kf\left(0\right)-f\left(X_{s^-}^u\right)\right)M(ds,du,d\theta,dk),
\end{align*}
where $M$ is a Poisson point measure on $\mathbb{R}_+\times\mathcal{U}\times \mathbb{R}_+\times \mathbb{N}$ with intensity $ds\otimes n(du)\otimes d\theta\otimes p(dk)$, where $p$ denotes the distribution of the number of descendants.

In order to get information on the average number of individuals in the population at time $t$, we follow Harris in \cite[Chapter 6]{harris} and we obtain
\[
m(0,0,t)=1-G(t)+m\int_0^tm(0,0,t-u)dG(u),
\]
where $m$ is the average number of descendants at division. From this expression, we can derive a renewal equation for $m(x,0,t)$, for $x\geq 0$.
% Using the life-time distribution conditioned to be greater than $x$ for the first individual 
%we have:
%\begin{equation}\label{eq:age}
%m(x,0,t)=\frac{m(0,0,t+x)-m\int_0^x m(0,0,t+x-u)dG(u)}{1-G(x)}.
%\end{equation}
We cannot find an explicit solution to this renewal equation except in the case of an exponentially distributed lifetime but we know the asymptotic behavior of a solution (see \cite{harris}). In particular, if $G$ is non-lattice and $m>1$,
% let $\alpha$ be the positive root of:
\begin{align*}
m(0,0,t)\underset{t\rightarrow +\infty}{\sim} c(\alpha,m)n_1 e^{\alpha t},\quad \text{where }\int_0^{\infty}e^{-\alpha t}dG(t)=1/m,
\end{align*}
and
%\begin{align*}
%c(\alpha,m)=\frac{m-1}{\alpha m^2\int_0^{\infty}te^{-\alpha t}dG(t)},
%\end{align*}
$c(\alpha,m)$ and $n_1$ are explicitly given in \cite[Theorem 17.1]{harris}. The rate of division of the auxiliary process is thus given for large $t$ by
\begin{align*}
\widehat{B}_s^{(t)}\sim B(x)\frac{e^{-\alpha x}(1-G(x))}{1-m\int_0^x e^{-\alpha u}dG(u)}.
\end{align*}
\subsection{Multi-type branching process and switching} An example of phenomenon that we would like to understand using a model on a finite state space is the phenotypic switching, i.e. the capacity to achieve multiple internal states in response to a single set of external inputs. Examples of studies of switching can be found in \cite{ozbudak2004multistability} or \cite{leibler2010individual}. For an asymptotic characterization of the ancestral lineage of a typical individual for models with a trait on a finite state space, we refer to \cite{georgii2003supercritical}. We assume here that an individual can be in state $0$ or $1$ which is constant during its lifetime. An individual in state $x=0,1$ divide at rate $B(x)=b_x$ and at division, it is replaced by $2$ individuals. We denote by $p$ the probability of switching at birth. We assume that this probability does not depend on the trait. Therefore, the trait only affects the lifetime of individuals. 
We obtain for the generator of the first moment semi-group for any function $f$ taking values in $\left\lbrace 0,1\right\rbrace$ and any $x\in\left\lbrace 0,1\right\rbrace$:
\begin{align*}
\mathcal{F}_{\text{switch}}f(x)=B(x)\left(2f(x)(1-p)+2f(\overline{x})p-f(x)\right),
\end{align*}
where $\overline{x}=1-x$.
%Moreover we have:
%\begin{align*}
%\mathbb{E}_{\delta_x}\left[\sum_{u\in V_t}B(X_t^u)\right]=(b_1-b_0)\mathbb{E}_{\delta_x}\left[\sum_{u\in V_t}X_t^u\right]+b_0\mathbb{E}_{\delta_x}\left[N_t\right].
%\end{align*}
After some computations, we obtain
\begin{align*}
\mathbb{E}_{\delta_x}\left[\sum_{u\in V_t}X_t^u\right]=x +(b_1(1-2p)-2pb_0)\int_0^t\mathbb{E}_{\delta_x}\left[\sum_{u\in V_s}X_s^u\right]ds+2pb_0\int_0^t\mathbb{E}_{\delta_x}\left[N_s\right]ds.
\end{align*}
Then, if we write: \[\mu(t)=\E_{\delta_x}\left[N_t\right],\ \nu(t)=\E_{\delta_x}\left[\sum_{u\in V_t}X_t^u\right],\ \forall t\geq 0,\]
we obtain
\[\left(\begin{array}{c}
\partial_t \mu\\
\partial_t \nu
\end{array}\right)
=
\left(\begin{array}{cc}
b_0&b_1-b_0\\
2pb_0& b_1(1-2p)-2pb_0
\end{array}\right)
\left(\begin{array}{c}
\mu\\
\nu\\
\end{array}\right).\]
For example, for $p=1/2$, writing $\gamma=\frac{b_0}{b_1}$, we have
\begin{align*}
m(1,s,t)=m(0,s,t)+\left[e^{\sqrt{b_0b_1}(t-s)}-e^{-\sqrt{b_0b_1}(t-s)}\right]\frac{1}{2\sqrt{\gamma}}(1-\gamma).
\end{align*} 
In particular, the transition kernel of the auxiliary process is given by
\begin{align*}
\widehat{P}_s^{(t)}(x,dy)=\frac{m(x,s,t)\delta_x(dy)+m(\overline{x},s,t)\delta_{\overline{x}}(dy)}{m(x,s,t)+m(\overline{x},s,t)},
\end{align*}
so that if $\gamma>1$, i.e. $b_0>b_1$, the auxiliary process switches more from $1$ to $0$ at a jump because $m(0,s,t)>m(1,s,t)$.
\subsection{Markovian jump processes for the dynamic of the trait} The dynamic of some characteristics of a cell are non-continuous and thus cannot be described by a diffusion type process. For example, this is the case for the dynamic of populations inside a cell such as plasmids or extra-chromosomal DNA. Then, an other generalization of Kimmel's multilevel model for plasmids \cite{kimmel1997quasistationarity} is the following: we assume that the trait of each individual evolves as a birth and death process with birth rate $\lambda>0$ and death rate $\mu>0$. We assume here that $\lambda-\mu>0$. The generator of the process corresponding to the dynamic of the trait is then given for any measurable function $f:\mathbb{N}\rightarrow \mathbb{R}_+$ and any $x\in \mathbb{N}$ by
\begin{align*}
\mathcal{G}f(x)=\lambda(f(x+1)-f(x))+\mu(f(x-1)-f(x)).
\end{align*}
We assume that a cell with $x$ plasmids divides at a rate $B(x)$ and that at division, the plasmids are randomly allocated to one of the two daughter cells. The branching process $(Z_t,t\geq 0)$ is solution of the following equation, for any measurable function $f:\mathbb{N}\rightarrow \mathbb{R}_+$ and any $x\in\mathcal{X}$,
\begin{multline*}
\langle Z_t,f\rangle =\langle Z_0,f\rangle +\int_0^t \int_{\mathbb{R}_+}\sum_{u\in V_s}\left[\mathbf{1}_{\left\{\theta\leq \lambda X_{s^-}^u\right\}}\left(f(X_{s^-}^u+1)-f(X_{s^-}^u)\right)\right.\\
\left. +\mathbf{1}_{ \left\{\lambda X_{s^-}^u\leq \theta\leq  (\lambda+\mu) X_{s^-}^u\right\}}\left(f(X_{s^-}^u-1)-f(X_{s^-}^u)\right)\right]Q^u(ds,d\theta)\\
+\int_0^t\int_{\mathcal{U}\times \mathbb{R}_+\times [0,1]}\mathbf{1}_{\left\{u\in V_{s^-},\ z\leq B\left( X_{s^-}^u\right)\right\}}\left(f\left(\delta X_{s^-}^u\right)+f\left((1-\delta) X_{s^-}^u\right)-f\left(X_{s^-}^u\right)\right)\\
\times M(ds,du,dz,d\delta),
\end{multline*}
where $\left(Q^u\right)_{u\in\mathcal{U}}$ is a family of Poisson point measure on $\mathbb{R}_+\times \mathbb{R}_+$ with intensity $ds\otimes d\theta$ and $M$ is a Poisson point measure on $\mathbb{R}_+\times\mathcal{U}\times \mathbb{R}_+\times [0,1]$ with intensity $ds\otimes n(du)\otimes dz\otimes d\delta$.

For example, for the division rate $B(x)=x$, we obtain for the average number of individuals in the population after a time $t$
\begin{align*}
m(x,s,t)=1+\frac{x}{\lambda-\mu}\left(e^{(\lambda-\mu)t}-1\right).
\end{align*}
In particular, the motion of the auxiliary process between jumps is given by the following generator:
\begin{align*}
\widehat{\mathcal{G}}_s^{(t)}f(x)=&\lambda\left[1+\frac{e^{(\lambda-\mu)(t-s)}-1}{\lambda-\mu+x\left(e^{(\lambda-\mu)(t-s)}-1\right)}\right]\left(f(x+1)-f(x)\right)\\
&+\mu\left[1-\frac{e^{(\lambda-\mu)(t-s)}-1}{\lambda-\mu+x\left(e^{(\lambda-\mu)(t-s)}-1\right)}\right]\left(f(x-1)-f(x)\right).
\end{align*}
The birth rate of the plasmid population for the auxiliary process is bigger than $\lambda$ and the death rate is smaller than $\mu$. This can be explained again by the fact that cells with a lot of plasmids divides more so that they are more represented at sampling. 
\subsection{Link with the integro-differential model}

The study of the average process associated with the measure-valued branching process $Z$ is interesting in the sense that it characterizes the macroscopic evolution of the population. For a more detailed study of this link see for example \cite{campillo2016links}. The following result is a corollary of Theorem \ref{th:existence} of Section \ref{sec:exis}. We recall that for all $s\geq 0$, $t\geq s$ and $x\in\mathcal{X}$,
\begin{align*}
R_{s,t}f(x)=\E\left[\sum_{u\in V_t} f\left(X_t^u\right)\middle|Z_s=\delta_x\right],
\end{align*}
where $f$ is a measurable function. 

\begin{cor} Let $f\in D\left(\mathcal{G}\right)$, $s\geq 0$ and $x_0\in \mathcal{X}$. Under Assumptions \ref{assu:debut1}(1-3) and \ref{assu:debut2}, the measure $\left(R_{s,t}(x_0,\cdot)\right)_{t\geq 0}$ is the unique solution to the following integro-differential equation:
\begin{multline}
R_{s,t}f(s,x_0)  =f\left(s,x_{0}\right)+\int_{s}^{t}\int_{\mathcal{X}}\left(\mathcal{G}f(r,x)+\partial_{r}f\left(r,x\right)\right)R_{s,r}(x_0,dx)ds \\
+\int_{s}^{t}\int_{\mathcal{X}}B(x)\left[\sum_{k\geq0}p_{k}(x)\sum_{j=1}^{k}\int_{\mathcal{X}}f\left(r,y\right)P_j^{(k)}\left(x,dy\right)-f\left(r,x\right)\right]R_{s,r}(x_0,dx)ds,\label{eq:moyenne}
\end{multline}
where $\left(R_{s,t}\right)_{t\geq s}$ is defined in \eqref{eq:firstmoment}.
\end{cor}

One can prove this result taking the expectation in \eqref{eq:evol} and using the same arguments as in the proof of Corollary 2.4 in \cite{cloez}.

Let $n(t,\cdot):=R_{0,t}(x_0,\cdot)$. Equation \eqref{eq:moyenne} can be rewritten as
\[
\begin{cases}
\partial_t n(t,x)  =\mathcal{G}^Tn(t,x)+\sum_{k\geq 0}\sum_{j=1}^k K_j^{(k)}\left(Bp_kn(t,\cdot)\right)-B(x)n(t,x),\\
n(0,x)dx=\delta_{x_0}(dx).
\end{cases}
\]
where $\mathcal{G}^T$, $\mathcal{G}$ are the adjoint operators of $K_j^{(k)}$ and $f\mapsto \int_{\mathcal{X}} f\left(y\right)P_j^{(k)}(x,dy)$ respectively, as in \cite{cloez}.

For example, in the case of the cell division model with exponential growth introduced in Section \ref{sub:binary}, we obtain in a weak sense
\begin{align*}
\partial_tn(t,x)+\partial_x\left(axn(t,x)\right)= 4B(2x)n(t,2x)-B(x)n(t,x). 
\end{align*}
This is a classical growth-fragmentation equation as the one studied in \cite{mischler2015spectral} or \cite{campillo2016links}. The solutions of the associated eigenvalue problem permit in particular to quantify the asymptotic global growth rate of the population.
\paragraph*{Acknowledgements.}I would like to thank Vincent Bansaye for his guidance during this work. I thank Marc Hoffmann and Bertrand Cloez for many helpful discussions on the subject of this paper. I also thank an anonymous referee for its useful comments. I acknowledge partial support by the Chaire Mod\'elisation Math\'ematique
et Biodiversit\'e of Veolia Environnement - \'Ecole Polytechnique - Museum National Histoire Naturelle - F.X. This work is supported by the "IDI 2014" project funded by the IDEX Paris-Saclay, (ANR-11-IDEX-0003-02) and by the French national research agency (ANR) via project MEMIP (ANR-16-CE33-0018).
\appendix
\section{Proof of Lemma \ref{lemma:existence}}\label{app:existence}
We give a recursive construction of the solution to \eqref{eq:evol}. For all $u\in \mathcal{U}$, we denote the birth time and the death time of $u$ respectively by $\alpha(u)$ and $\beta(u)$. Let $x_0\in\mathcal{X}$ be given. We construct a structured population $Y^k=\left(\bar{Z}^k,(X_s^u,\ s\geq T_k(\bar{Z}^k),\ u\in V_{T_k(\bar{Z}^k)})\right)$, where $\bar{Z}^k\in\mathbb{D}\left(\mathbb{R}_{+},\mathcal{M}_{P}\left(\mathcal{U}\times\mathcal{X}\right)\right)$ is such that $T_{k+1}=+\infty$. We set $\alpha(\emptyset)=0$, $X_0^{\emptyset}=x_0$, $V_0=\left\lbrace \emptyset\right\rbrace$ and $\bar{Z}^0_t\equiv\delta_{\left(\emptyset,x_0\right)}$ for all $t\geq 0$, so that
\[Y^0=(\bar{Z}^0,(\Phi^{\emptyset}(x_0,0,t), t\geq 0)).
\]
Let $k\geq 1$. We now construct $Y^{k+1}$. For all $u\in V_{T_k(\bar{Z}^k)}$ such that $\alpha(u)=T_k(\bar{Z}^k)$ and for all $t\geq \alpha(u)$, we set $X_t^u=\Phi^u(X_{\alpha(u)}^u,\alpha(u),t)$. For all $u\in V_{T_k(\bar{Z}^k)}$, let
\[
\beta(u)=\inf\left\lbrace t>\alpha(u),\ \int_{\alpha(u)}^t\int_{\mathbb{R}_+} \mathbf{1}_{\left\{z\leq B(X_{s^-}^{u})\right\}}M(ds,\left\{u\right\},dz,[0,1],[0,1])>0\right\}.
\]
Let $T=\inf\left\lbrace\beta(u),\ u\in V_{T_k(\bar{Z}^k)}\right\rbrace$.
Let $(T,U_{k+1},\theta_{k+1},L_{k+1},A_{k+1})$ be the unique quintuplet of random variables such that $M\left(\{T\},\{U_{k+1}\},\{\theta_{k+1}\},\{L_{k+1}\},\{A_{k+1}\}\right)=1$. Let
\[
V_{T}=V_{T^-}\setminus\left\lbrace U_{k+1}\right\rbrace\bigcup\left\lbrace U_{k+1}1,\ldots,U_{k+1}G(U_{k+1},T,L_{k+1})\right\rbrace,
\]
and for all $i=1,\ldots, G\left(X^{U_{k+1}}_T,L_{k+1}\right)$, we set $\alpha(U_{k+1}i)=T$ and
\[
X_{\alpha(U_{k+1}i)}^{U_{k+1}i} =F_i\left(X^{U_{k+1}}_T,L_{k+1},A_{k+1}\right).
\]
We set 
\begin{align*}
\bar{Z}^{k+1}_t&=\bar{Z}^k_t,\text{ for all }t\in [0,T_k(\bar{Z}^k)],\\
\bar{Z}^{k+1}_t&=\sum_{u\in V_{T_k(\bar{Z}^k)}}\delta_{\left(u,X_t^u\right)}, \text{ for all }t\in [T_k(\bar{Z}^k),T[,\\
\bar{Z}^{k+1}_t&=\sum_{u\in V_{T}}\delta_{\left(u,X_T^u\right)}, \text{ for all }t\geq T.
\end{align*}
Finally, we set $Y^{k+1}=(\bar{Z}^{k+1},\left(X_s^u,\ s\geq T, u\in V_T\right))$ so that $T_{k+1}(\bar{Z}^{k+1})=T$.

Let $\bar{Z}$ be the measure-valued branching process on $\mathbb{R}_+$ satisfying, for all $k\in\mathbb{N}$ and all $t\geq 0$,
\[
\bar{Z}_{t\wedge T_k(\bar{Z}^k)}=\bar{Z}_t^k.
\]
Therefore, $T_k(\bar{Z})=T_k(\bar{Z}^k)$ for all $k\in\mathbb{N}$. To shorten notation, we write $T_k$ instead of $T_k(\bar{Z})$ until the end of the proof.

Let $f\in \mathcal{D}(\mathcal{G})$. We now prove by induction the following property:
\begin{align*}
\mathcal{H}_k:\left\lbrace\forall t\in [T_k,T_{k+1}),\ \langle \bar{Z}_{t},f\rangle\text{ is a solution to \eqref{eq:evol}}.\right\rbrace
\end{align*}
First, $\mathcal{H}_0$ is obviously true. Assume that $\mathcal{H}_{k-1}$ is true. Let $t\in [T_k,T_{k+1})$. We recall that $U_k$ denotes the individual who dies at time $T_k$. We denote by
\begin{align*}
V_{t,1}=V_{T_{k-1}}\setminus\left\lbrace U_k\right\rbrace,\ V_{t,2}=\left\lbrace u\in V_t |\alpha(u)=T_k\right\rbrace,
\end{align*}
the set of all individuals born strictly before $T_k$ except $U_k$ and the descendants of $U_k$, respectively.
We have
\begin{align*}
\sum_{u\in V_t} f\left(u,t,X_t^u\right)=\sum_{u\in V_{t,1}}f\left(u,t,X_t^u\right)+\sum_{u\in V_{t,2}}f\left(u,t,X_t^u\right),
\end{align*}
and
\begin{align*}
f\left(u,t,X_t^u\right)=f\left(u,t,\Phi^u\left(X_{T_{k-1}},T_{k-1},t\right)\right).
\end{align*}
As none of the individuals in $V_{T_{k-1}}\setminus\left\lbrace U_k\right\rbrace$ divides on $[T_{k-1},t],$ using \eqref{eq:probmart}, we obtain
\begin{align*}
& f\left(u,t,X_t^u\right)=f\left(u,T_{k-1},X_{T_{k-1}}^u\right)+ \int_{T_{k-1}}^{t}\left(\mathcal{G}f(u,s,X_s^u)+\partial_s f\left(u,s,X_s^u\right)\right)ds\\
& \hspace{10cm}+ M_{T_{k-1},t}^{f,u}\left(X^u_{T_{k-1}}\right).
\end{align*}
Then, we split both the integral term and the martingale in two terms to separate the behavior of the population before $T_k$ and after $T_k$. We add and subtract the contribution corresponding to $U_k$ to get a sum over all individuals alive at time $T_{k-1}$:  
\begin{multline}\label{eq:ind1}
\sum_{u\in V_{t,1}}f\left(u,t,X_t^u\right)=\sum_{u\in V_{T_{k-1}}\setminus\left\lbrace U_k\right\rbrace} \int_{T_{k}}^{t}\left(\mathcal{G}f(u,s,X_s^u)+\partial_s f\left(u,s,X_s^u\right)\right)ds\\
+\sum_{u\in V_{T_{k-1}}\setminus\left\lbrace U_k\right\rbrace}M_{T_{k},t}^{f,u}\left(X^u_{T_{k}}\right)-f\left(U_k,T_k^-,X_{T_{k}^-}^{U_k}\right)+\sum_{u\in V_{T_{k-1}}}f\left(u,T_{k-1},X_{T_{k-1}}^u\right)\\
+ \sum_{u\in V_{T_{k-1}}}\int_{T_{k-1}}^{T_k}\left(\mathcal{G}f(u,s,X_s^u)+\partial_s f\left(u,s,X_s^u\right)\right)ds+\sum_{u\in V_{T_{k-1}}}M_{T_{k-1},T_k}^{f,u}\left(X^u_{T_{k-1}}\right).
\end{multline}
Using the induction hypothesis, we have
\begin{multline}\label{eq:ind2}
\sum_{u\in V_{T_{k-1}}}f\left(u,T_{k-1},X_{T_{k-1}}^u\right)=f\left(u,0,x_{0}\right)\\
+\int_{0}^{T_{k-1}}\int_{\mathcal{U}\times\mathcal{X}}\left(\mathcal{G}f(u,s,x)+\partial_s f\left(u,s,x\right)\right)\bar{Z}_{s}\left(dudx\right)ds+ M_{0,T_{k-1}}^f(x_0)\\
 +\int_{0}^{T_{k-1}}\int_{E}\mathbf{1}_{\left\{u\in V_{s^{-}},\ z\leq B\left(X_{s^{-}}^u\right)\right\}}
 \left(\sum_{i=1}^{G\left(X_s^u,l\right)}f\left(u,s,F_{i}\left(X^u_s,l,\theta\right)\right)-f\left(u,s,X_{s^{-}}^{u}\right)\right)\\
 \times M\left(ds,du,dz,dl,d\theta\right).
\end{multline}
Moreover, for all $s\in[T_{k-1},T_k[$, $V_s=V_{T_{k-1}}$, so that we have
\begin{align}\nonumber
\sum_{u\in V_{T_{k-1}}} \int_{T_{k-1}}^{T_k} & \left(\mathcal{G}f(u,s,X_s^u)+\partial_s f\left(u,s,X_s^u\right)\right)ds\\\label{eq:ind3}
& = \int_{T_{k-1}}^{T_k}\sum_{u\in V_{s}}\left(\mathcal{G}f(u,s,X_s^u)+\partial_s f\left(u,s,X_s^u\right)\right)ds.
\end{align}
Finally, combining \eqref{eq:ind1}, \eqref{eq:ind2} and \eqref{eq:ind3}, we obtain
\begin{multline}\label{eq:constr1}
\sum_{u\in V_{t,1}}f\left(u,t,X_t^u\right)=f\left(\emptyset,0,x_{0}\right)+\int_{0}^{T_{k}}\int_{\mathcal{U}\times\mathcal{X}}\left(\mathcal{G}f(u,s,x)+\partial_s f\left(u,s,x\right)\right)\bar{Z}_{s}\left(dudx\right)ds\\
+M_{0,T_{k-1}}^{f}(x_0)+\sum_{u\in V_{T_{k-1}}}M_{T_{k-1},T_k}^{f,u}\left(X^u_{T_{k-1}}\right)-f\left(U_k,T_k^-,X_{T_{k}^-}^{U_k}\right)\\
 +\int_{0}^{T_{k-1}}\int_{E}\mathbf{1}_{\left\{u\in V_{s^{-}},\ z\leq B\left(X_{s^{-}}^u\right)\right\}}
 \left(\sum_{i=1}^{G\left(X_s^u,l\right)}f\left(u,s,F_{i}\left(X_s^u,l,\theta\right)\right)-f\left(u,s,X_{s^{-}}^{u}\right)\right)\\
 \times M\left(ds,du,dz,dl,d\theta\right)\\
 +\sum_{u\in V_{T_{k-1}}\setminus\left\lbrace U_k\right\rbrace}\left[ \int_{T_{k}}^{t}\left(\mathcal{G}f(u,s,X_s^u)+\partial_s f\left(u,s,X_s^u\right)\right)ds+M_{T_{k},t}^{f,u}\left(X^u_{T_{k}}\right)\right].
\end{multline}
Next, using again \eqref{eq:probmart}, we have
\begin{align}\nonumber
\sum_{u\in V_{t,2}}& f\left(u,t,X_t^u\right)\\\label{eq:constr2}
&=\sum_{u\in V_{t,2}}\left[f(u,T_k,X_{T_k}^u)+\int_{T_{k}}^{t}\left(\mathcal{G}f(u,s,X_s^u)+\partial_s f\left(u,s,X_s^u\right)\right)ds+M_{T_k,t}^{f,u}\left(X^u_{T_{k}}\right)\right].
\end{align}
Moreover, by definition of $V_{t,2}$,
\begin{align}\nonumber
\sum_{u\in V_{t,2}}& f(u,T_k,X_{T_k}^u)\\\label{eq:constr3}
& =\int_{T_{k-1}}^t\int_E\mathbf{1}_{\left\{u\in V_{s^-},\ z\leq B(X_s^u)\right\}}\sum_{i=1}^{G(X_s^u,l)}f(u,s,F_i(X_s^u,l,\theta))M(ds,du,dz,dl,d\theta).
\end{align}
Adding the martingale terms of \eqref{eq:constr1} and \eqref{eq:constr2}, we obtain
\begin{align}\nonumber
M_{0,T_{k-1}}^f(x_0)& +\sum_{u\in V_{T_{k-1}}}M_{T_{k-1},T_k}^{f,u}\left(X^u_{T_{k-1}}\right)\\\label{eq:constr4}
&+\sum_{u\in V_{T_{k-1}}\setminus\left\lbrace U_k\right\rbrace} M_{T_{k-1},t}^{f,u}\left(X^u_{T_{k-1}}\right)+\sum_{u\in V_{T_{k}}, \alpha(u)=T_k} M_{T_k,t}^{f,u}\left(X^u_{T_{k}}\right) =M_{0,T_k}^f(x_0).
\end{align}
Finally, we obtain the result combining \eqref{eq:constr2},\eqref{eq:constr3} and \eqref{eq:constr4}.
\section{Proof of Lemma \ref{lemma:uniqueness}}\label{app:uniqueness}
Let $\bar{Z}^{(1)}$ and $\bar{Z}^{(2)}$ be two solutions of \eqref{eq:evol} associated with the previously defined family of flows and Poisson point measure. For all $k\in\mathbb{N}$, we write $T^{(i)}_{k}=T_k(\bar{Z}^{(i)})$, $i=1,2$. We assume that $\bar{Z}^{(1)}_0=\bar{Z}^{(2)}_0=\delta_{x} $, for some $x\in\mathcal{X}$. We have $T^{(1)}_{0}=T_{0}^{(2)}=0$. We prove by induction on $k\in\mathbb{N}$ the following proposition:
\begin{align*}
\mathcal{H}_k:\ T^{(1)}_{k+1}=T^{(2)}_{k+1} \text{ and } \forall t\in [T_k^{(1)},T_{k+1}^{(1)}),\ \forall f\in \bar{\mathcal{D}}(\mathcal{G}),\ \langle \bar{Z}^{(1)}_t,f\rangle=\langle \bar{Z}^{(2)}_t,f\rangle.
\end{align*}
First, $\mathcal{H}_0$ is true because
\begin{align*}
T^{(1)}_{1}=T^{(2)}_{1}=\inf\left\lbrace t>0,\ \int_{0}^t\int_{\mathbb{R}_+} \mathbf{1}_{\left\{z\leq B(\Phi^{\emptyset}(x,0,s))\right\}}M(ds,\left\{\emptyset\right\},dz,[0,1],[0,1])>0\right\}
\end{align*}
and for all $t\in  [0,T_1)$, $\bar{Z}^{(1)}_t=\bar{Z}^{(2)}_t=\delta_{\Phi^{\emptyset}(x,0,t)}$. 
Let us assume that $\mathcal{H}_{k-1}$ is true. We first prove the second point of $\mathcal{H}_k$. By \eqref{eq:evol}, for $i=1,2$, we have
\begin{multline}
\langle \bar{Z}^{(i)}_{T_k^{(1)}},f\rangle=\langle \bar{Z}^{(i)}_{T_{k-1}^{(1)}},f\rangle\\
+\int^{T_k^{(1)}}_{T_{k-1}^{(1)}}\int_{\mathcal{U}\times\mathcal{X}}\left(\mathcal{G}f(u,s,x)+\partial_s f\left(u,s,x\right)\right)\bar{Z}^{(i)}_{s}\left(du,dx\right)ds+ M_{T_{k-1}^{(1)},T_k^{(1)}}^{f,(i)}(x)\\
 +\int^{T_k^{(1)}}_{T_{k-1}^{(1)}}\int_{E}\mathbf{1}_{\left\{u\in V^{(i)}_{s^{-}},\ z\leq B\left(X_{s^{-}}^{u,(i)}\right)\right\}}
 \Big(\sum_{i=1}^{G\left(X_s^u,l\right)}f\left(u,s,F_{i}\left(X_s^u,l,\theta\right)\right)-f\Big(u,s,X_{s^{-}}^{u,(i)}\Big)\Big)\\\label{eq:jump}
M\left(ds,du,dz,dl,d\theta\right).
\end{multline}
As the jump integral \eqref{eq:jump} depends only on the process strictly before $T_k^{(1)}$, we obtain, using the induction hypothesis, that $\langle \bar{Z}^{(1)}_{T_k^{(1)}},f\rangle=\langle \bar{Z}^{(2)}_{T_k^{(1)}},f\rangle$. The evolution of the trait for $t\in \left[T_k^{(1)},T_{k+1}^{(1)}\wedge T_{k+1}^{(2)}\right)$ only depends on the family of flows given at the beginning and which are the same for both solutions. Hence, it remains to prove that $T_{k+1}^{(1)}= T_{k+1}^{(2)}$. And it is the case because this jump time only depend on the state of the population at $T_k^{(1)}$, on the flows and on the Poisson point measure $M$.
Finally, for all $t\in \left[T_k^{(1)}, T_{k+1}^{(1)}\right)$, we have: $\langle \bar{Z}^{(1)}_{t},f\rangle=\langle \bar{Z}^{(2)}_{t},f\rangle$.

Moreover, the measure-valued process is entirely characterized by $\left\lbrace\langle \bar{Z}_t,f\rangle, f\in\bar{\mathcal{D}}(\mathcal{G})\right\rbrace$ according to Remark \ref{rem:densite}. Therefore, there is a unique c\`adl\`ag measure-valued strong solution to \eqref{eq:evol} up to the $k$th jump time for all $k\in\mathbb{N}$. 
\section{Proof of Lemma \ref{lemma:mdansG}}\label{app:mdansG}
Let $t\geq 0$ and $s\leq t$. We want to prove that $m(\cdot,s, t)\in\mathcal{D}(\mathcal{G})$ i.e. that $$\lim_{r\downarrow 0}\frac{\mathbb{E}(m(X_r,s,t)\big|X_0=x)-m(x,s,t)}{r}=\lim_{r\downarrow s}\frac{\mathbb{E}(m(X_r,s,t)\big|X_s=x)-m(x,s,t)}{r-s}$$ exists. Let $s\leq r <t$ and $x\in\mathcal{X}$. We consider the event $\Omega_r=\left\lbrace \text{no division before } r\right\rbrace$. Then
\begin{align*}
\frac{1}{r-s}\left(\mathbb{E}(m(X_{r},s,t)\big|X_s=x)-m(x,s,t)\right)=A(x,r,s,t)+B(x,r,s,t)+C(x,r,s,t),
\end{align*} where
\begin{align*}
A(x,r,s,t)&=\frac{1}{r-s}\left(\mathbb	E (m(X_{r},s,t)\big|X_s=x)-\mathbb{E}(<Z_{r},m(\cdot,s,t)>\mathbf{1}_{\Omega_{r}}\big|Z_s=\delta_x)\right),\\
B(x,r,s,t)&=\frac{1}{r-s}\mathbb{E}(<Z_{r},m(\cdot,s,t)>\left(\mathbf{1}_{\Omega_{r}}-1\right)\big|Z_s=\delta_x),\\
C(x,r,s,t)&=\frac{1}{r-s}\left(\mathbb{E}(<Z_{r},m(\cdot,s,t)>\big|Z_s=\delta_x)-m(x,s,t)\right).
\end{align*}
First, 
\begin{align*}
A(x,r,s,t)=\frac{1}{r-s}\left(\mathbb{E}\left(m(X_{r},s,t\big|X_s=x)\right)-\mathbb{E}\left(m(X_{r}^{\emptyset},s,t)\mathbf{1}_{\Omega_{r}}\middle|X_s^{\emptyset}=x\right)\right).
\end{align*}
Conditioning with respect to $\sigma(X_u,r\leq u\leq s)$ we obtain
\begin{align*}
A(x,r,s,t)=\frac{1}{r-s}\mathbb{E}\left(m(X_{r},s,t)\left(1-e^{-\int_s^{r}B(X_u)du}\right)\middle|X_s = x\right)\xrightarrow[r\rightarrow s]{} m(x,s,t)B(x).
\end{align*} 
Next, we have
\begin{align*}
B(x,r,s,t)&=-\frac{1}{r-s}\mathbb{E}(<Z_{r},m(\cdot,s,t)>\mathbf{1}_{\Omega_{r}^C}\big|Z_s=\delta_x).
\end{align*}
Then, let us denote $T_1$ the random variable corresponding to the lifetime of the first individual. Using the Markov property and the branching property, we have
\begin{multline*}
B(x,r,s,t)=-\frac{1}{r-s}\mathbb{E}\left(\sum_{u\in V_{r}}m(X_{r}^u,s,t)\mathbf{1}_{\{T_1<r\}}\middle|Z_s=\delta_x\right)\\
=-\frac{1}{r-s}\mathbb{E}\Bigg(\mathbf{1}_{\{T_1<r\}}\sum_{k\geq 0}p_k\left(X_{T_1}^{\emptyset}\right)
 \sum_{j=0}^k\int_0^1\mathbb{E}\left(\sum_{u\in V_{r}}m(X_{r}^u,s,t)\big|X_{T_1}^u = F_j^{(k)}(X_{T_1}^{\emptyset},\theta)\right)\\
 \times P_j^{(k)}(X_{T_1}^{\emptyset},d\theta)\Big|Z_s=\delta_x\Bigg).
\end{multline*}
Next, exhibiting the distribution of $T_1$ we obtain
\begin{align*}
B(x,r,s,t)=& -\frac{1}{r-s}\int_s^{r}\mathbb{E}\left(B\left(X_{v}^{\emptyset}\right)e^{-\int_s^v B\left(X_l^{\emptyset}\right)dl}\sum_{k\geq 0}p_k\left(X_{v}^{\emptyset}\right)\right.\\
&\left. \sum_{j=0}^k\int_0^1\mathbb{E}\left(\sum_{u\in V_{r}}m(X_{r}^u,s,t)\middle|X_{v}^u = F_j^{(k)}(X_{v}^{\emptyset},\theta)\right)P_j^{(k)}(X_{v}^{\emptyset},d\theta)\middle|Z_s=\delta_x\right).
\end{align*}
Finally,
\begin{align*}
B(x,r,s,t)\xrightarrow[r\rightarrow s]{} &-B(x)\sum_{k\geq 0}p_k(x)\sum_{j=1}^k \int_0^1 m\left(F_j^{(k)}(x,\theta),s,t\right)P_j^{(k)}(x,d\theta).
\end{align*}
For the last term, we have
\begin{align*}
C(x,r,s,t)&=\frac{1}{r-s}\left(\mathbb{E}(<Z_{r},m(\cdot,s,t)>\big|Z_s=\delta_x)-m(x,s,t)\right)\\
& = \frac{1}{r-s}\left(R_{s,r}(R_{s,t}\mathbf{1})(x)-R_{s,t}\mathbf{1}(x)\right)\\
& = \frac{1}{r-s}R_{s,r}\left(R_{s,t}\mathbf{1}-R_{r,t}\mathbf{1}\right)(x)\\
& = R_{s,r}\left(\frac{m(\cdot,s,t)-m(\cdot,r,t)}{r-s}\right)(x)\xrightarrow[r\rightarrow s]{}-\partial_s m(x,s,t),
\end{align*}
because according to the first point of Assumption \ref{assu:differentiable}, $h^{-1}(m(x,s+h,t)-m(x,s,t))$ converges uniformly for $x$ in compact sets when $h$ tends to zero.
Finally, 
\begin{align*}
\lim_{r\rightarrow s}\frac{\mathbb{E}(m(X_r,s,t)\big|X_s=x)-m(x,s,t)}{r-s},
\end{align*}
exists.
\section{Details of the proof of Theorem \ref{th:mto}}\label{app:monotone}
We detail here the use of the monotone-class theorem in the proof of Theorem \ref{th:mto}.

Using Remark \ref{rem:densite}, \eqref{mtomesurable} is satisfied for any function of the form $F=\mathbf{1}_{B_1}\ldots\mathbf{1}_{B_n}$, where $B_i$ are Borel sets, for $i=1\ldots n$, for all $n\in\mathbb{N}$. Let us define
\begin{align*}
\mathcal{H}=\left\lbrace F:\mathbb{D}\left([0,t],\mathcal{X}\right)\rightarrow \mathbb{R}_+ \text{ bounded and measurable satisfying }\eqref{mtomesurable} \right\rbrace,
\end{align*}
and
\begin{align*}
I=\left\lbrace\bigcap_{i=1}^n \left\lbrace x\in \mathbb{D}\left([0,t],\mathcal{X}\right),\ x(s_i)\in B_i\right\rbrace , n\in\mathbb{N},\ s_i\in\mathbb{R}_+,\ B_i \text{ Borel sets}\right\rbrace.
\end{align*}
First, $I$ is a $\pi$-system and $\sigma(I)=\mathcal{D}$ where $\mathcal{D}$ is the Borel $\sigma$-field associated with the Skorokod topology on $\mathbb{D}([0,t],\mathcal{X})$ (\cite[Theorem 12.5]{billingsley2013convergence}). Then, applying the monotone-class theorem (\cite[Theorem 3.14]{williams1991probability}), we obtain that $\mathcal{H}$ contains all bounded measurable functions with respect to the Skorokhod topology.
\bibliographystyle{plain}
\bibliography{bibli}

\begin{thebibliography}{10}

\bibitem{athreya2000change}
K.~B. Athreya.
\newblock Change of measures for {M}arkov chains and the {$L\log L$} theorem
  for branching processes.
\newblock {\em Bernoulli}, 6(2):323--338, 2000.

\bibitem{athreya2012coalescence}
K.~B. Athreya.
\newblock Coalescence in critical and subcritical {G}alton-{W}atson branching
  processes.
\newblock {\em J. Appl. Probab.}, 49(3):627--638, 2012.

\bibitem{athreya2011}
K.~B. Athreya, S.~R. Athreya, and S.~K. Iyer.
\newblock Supercritical age-dependent branching {M}arkov processes and their
  scaling limits.
\newblock {\em Bernoulli}, 17(1):138--154, 2011.

\bibitem{bansaye2013}
V.~Bansaye.
\newblock Ancestral lineages and limit theorems for branching markov chains in
  varying environment.
\newblock {\em Journal of Theoretical Probability}, 2018.

\bibitem{bansaye2011limit}
V.~Bansaye, J.-F. Delmas, L.~Marsalle, and V.~C. Tran.
\newblock Limit theorems for {M}arkov processes indexed by continuous time
  {G}alton-{W}atson trees.
\newblock {\em Ann. Appl. Probab.}, 21(6):2263--2314, 2011.

\bibitem{bansaye2015stochastic}
V.~Bansaye and S.~M\'el\'eard.
\newblock {\em Stochastic models for structured populations}, volume~1 of {\em
  Mathematical Biosciences Institute Lecture Series. Stochastics in Biological
  Systems}.
\newblock Springer, Cham; MBI Mathematical Biosciences Institute, Ohio State
  University, Columbus, OH, 2015.
\newblock Scaling limits and long time behavior.

\bibitem{BT}
V.~Bansaye and V.~C. Tran.
\newblock Branching {F}eller diffusion for cell division with parasite
  infection.
\newblock {\em ALEA Lat. Am. J. Probab. Math. Stat.}, 8:95--127, 2011.

\bibitem{bellman1952age}
R.~Bellman and T.~Harris.
\newblock On age-dependent binary branching processes.
\newblock {\em Ann. of Math. (2)}, 55:280--295, 1952.

\bibitem{biggins1977martingale}
J.~D. Biggins.
\newblock Martingale convergence in the branching random walk.
\newblock {\em J. Appl. Probability}, 14(1):25--37, 1977.

\bibitem{biggins2004measure}
J.~D. Biggins and A.~E. Kyprianou.
\newblock Measure change in multitype branching.
\newblock {\em Adv. in Appl. Probab.}, 36(2):544--581, 2004.

\bibitem{billingsley2013convergence}
P.~Billingsley.
\newblock {\em Convergence of probability measures}.
\newblock Wiley Series in Probability and Statistics: Probability and
  Statistics. John Wiley \& Sons, Inc., New York, second edition, 1999.
\newblock A Wiley-Interscience Publication.

\bibitem{campillo2016links}
F.~Campillo, N.~Champagnat, and C.~Fritsch.
\newblock Links between deterministic and stochastic approaches for invasion in
  growth-fragmentation-death models.
\newblock {\em J. Math. Biol.}, 73(6-7):1781--1821, 2016.

\bibitem{chauvin1988kpp}
B.~Chauvin and A.~Rouault.
\newblock K{PP} equation and supercritical branching {B}rownian motion in the
  subcritical speed area. {A}pplication to spatial trees.
\newblock {\em Probab. Theory Related Fields}, 80(2):299--314, 1988.

\bibitem{chauvin1991growing}
B.~Chauvin, A.~Rouault, and A.~Wakolbinger.
\newblock Growing conditioned trees.
\newblock {\em Stochastic Process. Appl.}, 39(1):117--130, 1991.

\bibitem{cloez}
B.~Cloez.
\newblock Limit theorems for some branching measure-valued processes.
\newblock {\em Adv. in Appl. Probab.}, 49(2):549--580, 2017.

\bibitem{cooper2006distinguishing}
S.~Cooper.
\newblock Distinguishing between linear and exponential cell growth during the
  division cycle: single-cell studies, cell-culture studies, and the object of
  cell-cycle research.
\newblock {\em Theoretical Biology and Medical Modelling}, 3(1):10, 2006.

\bibitem{del2004feynman}
P.~Del~Moral.
\newblock {\em Feynman-{K}ac formulae}.
\newblock Probability and its Applications (New York). Springer-Verlag, New
  York, 2004.
\newblock Genealogical and interacting particle systems with applications.

\bibitem{doumic2015statistical}
M.~Doumic, M.~Hoffmann, N.~Krell, and L.~Robert.
\newblock Statistical estimation of a growth-fragmentation model observed on a
  genealogical tree.
\newblock {\em Bernoulli}, 21(3):1760--1799, 2015.

\bibitem{doumic2010}
M.~Doumic-Jauffret, P.~Maia, and J.~Zubelli.
\newblock On the calibration of a size-structured population model from
  experimental data.
\newblock {\em {Acta Biotheoretica}}, 58(4):405--413, 2010.

\bibitem{durrett1978genealogy}
R.~Durrett.
\newblock The genealogy of critical branching processes.
\newblock {\em Stochastic Process. Appl.}, 8(1):101--116, 1978/79.

\bibitem{ethier2009markov}
S.~N. Ethier and T.~G. Kurtz.
\newblock {\em Markov processes}.
\newblock Wiley Series in Probability and Mathematical Statistics: Probability
  and Mathematical Statistics. John Wiley \& Sons, Inc., New York, 1986.
\newblock Characterization and convergence.

\bibitem{fournier2004microscopic}
N.~Fournier and S.~M\'el\'eard.
\newblock A microscopic probabilistic description of a locally regulated
  population and macroscopic approximations.
\newblock {\em Ann. Appl. Probab.}, 14(4):1880--1919, 2004.

\bibitem{georgii2003supercritical}
H.-O. Georgii and E.~Baake.
\newblock Supercritical multitype branching processes: the ancestral types of
  typical individuals.
\newblock {\em Adv. in Appl. Probab.}, 35(4):1090--1110, 2003.

\bibitem{guyon2007limit}
J.~Guyon.
\newblock Limit theorems for bifurcating {M}arkov chains. {A}pplication to the
  detection of cellular aging.
\newblock {\em Ann. Appl. Probab.}, 17(5-6):1538--1569, 2007.

\bibitem{hardy2009spine}
R.~Hardy and S.~C. Harris.
\newblock A spine approach to branching diffusions with applications to
  {$L^p$}-convergence of martingales.
\newblock In {\em S\'eminaire de {P}robabilit\'es {XLII}}, volume 1979 of {\em
  Lecture Notes in Math.}, pages 281--330. Springer, Berlin, 2009.

\bibitem{harris2017coalescent}
S.~Harris, S.~Johnston, and M.~Roberts.
\newblock The coalescent structure of continuous-time {G}alton-{W}atson trees.
\newblock {\em ArXiv:1707.07993}, 2017.

\bibitem{harris2015many}
S;~C. Harris and M.~I. Roberts.
\newblock The many-to-few lemma and multiple spines.
\newblock {\em Ann. Inst. Henri Poincar\'e Probab. Stat.}, 53(1):226--242,
  2017.

\bibitem{harris1996large}
S.~C. Harris and D.~Williams.
\newblock Large deviations and martingales for a typed branching diffusion.
  {I}.
\newblock {\em Ast\'erisque}, 236:133--154, 1996.
\newblock Hommage \`a P. A. Meyer et J. Neveu.

\bibitem{harris}
T.~E. Harris.
\newblock {\em The theory of branching processes}.
\newblock Dover Phoenix Editions. Dover Publications, Inc., Mineola, NY, 2002.
\newblock Corrected reprint of the 1963 original [Springer, Berlin; MR0163361
  (29 \#664)].

\bibitem{hoang2015estimating}
V.~H. Hoang.
\newblock Estimating the division kernel of a size-structured population.
\newblock {\em ESAIM Probab. Stat.}, 21:275--302, 2017.

\bibitem{hoffmann2015nonparametric}
M.~Hoffmann and A.~Olivier.
\newblock Nonparametric estimation of the division rate of an age dependent
  branching process.
\newblock {\em Stochastic Process. Appl.}, 126(5):1433--1471, 2016.

\bibitem{hong2011coalescence}
J.-I Hong.
\newblock {\em Coalescence in {B}ellman-{H}arris and multi-type branching
  processes}.
\newblock ProQuest LLC, Ann Arbor, MI, 2011.
\newblock Thesis (Ph.D.)--Iowa State University.

\bibitem{kallenberg1977stability}
O.~Kallenberg.
\newblock Stability of critical cluster fields.
\newblock {\em Math. Nachr.}, 77:7--43, 1977.

\bibitem{kimmel1997quasistationarity}
M.~Kimmel.
\newblock Quasistationarity in a branching model of division-within-division.
\newblock In {\em Classical and modern branching processes}, volume~84 of {\em
  IMA Vol. Math. Appl.}, pages 157--164. Springer, New York, 1997.

\bibitem{kingman}
J.~F.~C. Kingman.
\newblock The first birth problem for an age-dependent branching process.
\newblock {\em Ann. Probability}, 3(5):790--801, 1975.

\bibitem{kunita1997stochastic}
H.~Kunita.
\newblock {\em Stochastic flows and stochastic differential equations},
  volume~24 of {\em Cambridge Studies in Advanced Mathematics}.
\newblock Cambridge University Press, Cambridge, 1997.
\newblock Reprint of the 1990 original.

\bibitem{kurtz1997conceptual}
T.~Kurtz, R.~Lyons, R.~Pemantle, and Y.~Peres.
\newblock A conceptual proof of the {K}esten-{S}tigum theorem for multi-type
  branching processes.
\newblock In {\em Classical and modern branching processes}, pages 181--185.
  Springer, 1997.

\bibitem{lambert2013coalescent}
A.~Lambert and L.~Popovic.
\newblock The coalescent point process of branching trees.
\newblock {\em Ann. Appl. Probab.}, 23(1):99--144, 2013.

\bibitem{leibler2010individual}
S.~Leibler and E.~Kussell.
\newblock Individual histories and selection in heterogeneous populations.
\newblock {\em Proc. Natl. Acad. Sci. USA}, 107(29):13183--13188, 2010.

\bibitem{lyons1995conceptual}
R.~Lyons, R.~Pemantle, and Y.~Peres.
\newblock Conceptual proofs of {$L\log L$} criteria for mean behavior of
  branching processes.
\newblock {\em Ann. Probab.}, 23(3):1125--1138, 1995.

\bibitem{marguet2017law}
A.~Marguet.
\newblock A law of large numbers for branching markov processes by the
  ergodicity of ancestral lineages.
\newblock {\em ArXiv:1707.07993}, 2017.

\bibitem{mischler2015spectral}
S.~Mischler and J.~Scher.
\newblock Spectral analysis of semigroups and growth-fragmentation equations.
\newblock {\em Ann. Inst. H. Poincar\'e Anal. Non Lin\'eaire}, 33(3):849--898,
  2016.

\bibitem{nerman1984stable}
O.~Nerman and P.~Jagers.
\newblock The stable double infinite pedigree process of supercritical
  branching populations.
\newblock {\em Z. Wahrsch. Verw. Gebiete}, 65(3):445--460, 1984.

\bibitem{o1995genealogy}
N.~O'Connell.
\newblock The genealogy of branching processes and the age of our most recent
  common ancestor.
\newblock {\em Adv. in Appl. Probab.}, 27(2):418--442, 1995.

\bibitem{ozbudak2004multistability}
E.~M. Ozbudak, M.~Thattai, H.~N. Lim, B.~I. Shraiman, and A.~Van~Oudenaarden.
\newblock Multistability in the lactose utilization network of {E}scherichia
  coli.
\newblock {\em Nature}, 427(6976):737--740, 2004.

\bibitem{tran}
V.~C. Tran.
\newblock {\em {Stochastic particle models for problems of adaptive evolution
  and for the approximations of statistical solutions}}.
\newblock Thesis, Universit{\'e} de Nanterre - Paris X, 2006.

\bibitem{williams1991probability}
D.~Williams.
\newblock {\em Probability with martingales}.
\newblock Cambridge Mathematical Textbooks. Cambridge University Press,
  Cambridge, 1991.

\bibitem{zubkov1976limiting}
A.~M. Zubkov.
\newblock Limit distributions of the distance to the nearest common ancestor.
\newblock {\em Teor. Verojatnost. i Primenen.}, 20(3):614--623, 1975.

\end{thebibliography}
\end{document}